\documentclass{amsart}

\usepackage[breaklinks,bookmarks=false]{hyperref}
\hypersetup{colorlinks,linkcolor=blue,citecolor=blue,urlcolor=blue,plainpages=false}

\usepackage[foot]{amsaddr}

\usepackage{amsmath}
\usepackage{amsthm}
\usepackage{mathrsfs}

\usepackage{tikz}
\usepackage{pgfplots}
\newcommand{\SlopeTriangle}[6]
{

    \pgfplotsextra
    {
        \pgfkeysgetvalue{/pgfplots/xmin}{\xmin}
        \pgfkeysgetvalue{/pgfplots/xmax}{\xmax}
        \pgfkeysgetvalue{/pgfplots/ymin}{\ymin}
        \pgfkeysgetvalue{/pgfplots/ymax}{\ymax}

        \pgfmathsetmacro{\xArel}{#1}
        \pgfmathsetmacro{\yArel}{#3}
        \pgfmathsetmacro{\xBrel}{#1-#2}
        \pgfmathsetmacro{\yBrel}{\yArel}
        \pgfmathsetmacro{\xCrel}{\xArel}

        \pgfmathsetmacro{\lnxB}{\xmin*(1-(#1-#2))+\xmax*(#1-#2)} 
        \pgfmathsetmacro{\lnxA}{\xmin*(1-#1)+\xmax*#1} 
        \pgfmathsetmacro{\lnyA}{\ymin*(1-#3)+\ymax*#3} 
        \pgfmathsetmacro{\lnyC}{\lnyA+#4*(\lnxA-\lnxB)}
        \pgfmathsetmacro{\yCrel}{\lnyC-\ymin)/(\ymax-\ymin)} 

        \coordinate (A) at (rel axis cs:\xArel,\yArel);
        \coordinate (B) at (rel axis cs:\xBrel,\yBrel);
        \coordinate (C) at (rel axis cs:\xCrel,\yCrel);

        \draw[#6]   (A)-- node[anchor=north] {#5}
                    (B)--
                    (C)--
                    cycle;
    }
}

\usepackage{a4wide}

\usepackage{amssymb}
\usepackage{mathrsfs}

\newcommand{\eq}{:=}

\newcommand{\grad}{\boldsymbol \nabla}
\renewcommand{\div}{\grad \cdot}

\newcommand{\BA}{\boldsymbol A}

\newcommand{\BI}{\boldsymbol I}

\newcommand{\BM}{\boldsymbol M}

\newcommand{\ba}{\boldsymbol a}
\newcommand{\bb}{\boldsymbol b}

\newcommand{\bd}{\boldsymbol d}
\newcommand{\be}{\boldsymbol e}

\newcommand{\bm}{\boldsymbol m}
\newcommand{\bn}{\boldsymbol n}

\newcommand{\bx}{\boldsymbol x}
\newcommand{\by}{\boldsymbol y}

\newcommand{\CE}{\mathcal E}
\newcommand{\CF}{\mathcal F}

\newcommand{\CO}{\mathcal O}
\newcommand{\CP}{\mathcal P}

\newcommand{\LH}{\mathscr H}

\newcommand{\LK}{\mathscr K}

\newcommand{\TA}{\textup A}

\newcommand{\N}{\mathbb N}
\newcommand{\Z}{\mathbb Z}
\newcommand{\R}{\mathbb R}
\newcommand{\C}{\mathbb C}

\newcommand{\bzero}{\boldsymbol 0}

\newcommand{\tf}{\widetilde f}
\newcommand{\tchi}{\widetilde \chi}

\newcommand{\cutoff}{\chi}

\newcommand{\can}[1]{\be_{#1}}

\newcommand{\bxx}{\boldsymbol{ x }}
\newcommand{\bxi}{\boldsymbol{\xi}}

\newcommand{\xxhm}{\bxx^{\hbar,\bm}}
\newcommand{\xihn}{\bxi^{\hbar,\bn}}

\newcommand{\pp}{{\rm p}}

\newcommand{\dist}{\operatorname{dist}}

\newcommand{\GS }[3]{\Phi_{#1,#2,#3}}

\newcommand{\GShknot}{\GS{\hbar}{\bx_0}{\bxi_0}}
\newcommand{\GShknotp}{\GS{\hbar}{\bx_0'}{\bxi_0'}}

\newcommand{\CIB}{C^\infty_{\rm b}}

\newcommand{\gs}[3]{\Psi_{#1,#2,#3}}
\newcommand{\gskmn}{\gs{k}{\bm}{\bn}}
\newcommand{\gskmnp}{\gs{k}{\bm'}{\bn'}}

\newcommand{\gshmn}{\gs{\hbar}{\bm}{\bn}}
\newcommand{\gshmnp}{\gs{\hbar}{\bm'}{\bn'}}

\newcommand{\xxkm}{\bx^{k,\bm}}

\newcommand{\xikn}{\bxi^{k,\bn}}

\newcommand{\dx}{\mathrm{d}\bx}
\newcommand{\dy}{\mathrm{d}\by}

\newcommand{\balpha}{\boldsymbol{\alpha}}
\newcommand{\bbeta}{\boldsymbol{\beta}}
\newcommand{\bgamma}{\boldsymbol{\gamma}}
\newcommand{\bdelta}{\boldsymbol{\delta}}

\newcommand{\tr}{\operatorname{tr}}

\newcommand{\HH}{\widehat{H}}

\newtheorem{theorem}{Theorem}
\newtheorem{lemma}[theorem]{Lemma}
\newtheorem{corollary}[theorem]{Corollary}
\newtheorem{proposition}[theorem]{Proposition}
\newtheorem{remark}[theorem]{Remark}

\numberwithin{equation}{section}
\numberwithin{theorem}{section}

\begin{document}

\title{Efficient approximation of high-frequency Helmholtz solutions by Gaussian coherent states}
\author{T. Chaumont-Frelet$^\dagger$}
\author{V. Dolean$^\ddagger$}
\author{M. Ingremeau$^\sharp$}

\address{\vspace{-.5cm}}
\address{\noindent \tiny \textup{$^\dagger$University C\^ote d'Azur, Inria, CNRS, LJAD}}
\address{\noindent \tiny \textup{$^\ddagger$University C\^ote d'Azur, CNRS, LJAD and University of Strathclyde}}
\address{\noindent \tiny \textup{$^\sharp$University C\^ote d'Azur, CNRS, LJAD}}

\maketitle
\thispagestyle{empty}

\begin{abstract}
We introduce new finite-dimensional spaces specifically designed to approximate the solutions
to high-frequency Helmholtz problems with smooth variable coefficients in dimension $d$. These
discretization spaces are spanned by Gaussian coherent states, that have the key property to be
localised in phase space. We carefully select the Gaussian coherent states spanning the
approximation space by exploiting the (known) micro-localisation properties of the solution.
For a large class of source terms (including plane-wave scattering problems), this choice leads
to discrete spaces that provide a uniform approximation error for all wavenumber $k$ with a number
of degrees of freedom scaling as $k^{d-1/2}$, which we rigorously establish. In comparison, for
discretization spaces based on (piecewise) polynomials, the number of degrees of freedom has to
scale at least as $k^d$ to achieve the same property. These theoretical results are illustrated by
one-dimensional numerical examples, where the proposed discretization spaces are coupled with
a least-squares variational formulation.

\vspace{.2cm}

\end{abstract}

\section{Introduction}

Time-harmonic wave propagation is a mechanism at the center of a large number of
physical and industrial applications. We may cite, among many, radar imaging \cite{dorf_2006a},
or seismic prospection \cite{tarantola_1984a}. In practice, numerical methods are required to
approximately simulate the propagation of waves, and although several methods are available,
it is still very challenging to compute accurate approximations in
the high-frequency regime.

Here, we consider the scalar Helmholtz equation, which is probably the simplest model
for this kind of problems. Specifically, given a compactly supported right-hand side
$f: \R^d \to \C$, our model problem is to find $u: \R^d \to \C$ such that
\begin{subequations}
\label{eq_helmholtz_intro}
\begin{equation}
\label{eq_volume_intro}
-k^2 \mu u-\div \left(\BA \grad u\right) = f \text{ in } \R^d,
\end{equation}
where $\mu$ and $\BA$ are (given) smooth coefficients that are respectively equal to
$1$ and $\BI$ outside a ball of radius $R > 0$, and $k > 0$ is the (given) wavenumber.
This equation is supplemented with the Sommerfeld radiation condition at infinity. Namely,
we require that
\begin{equation}
\label{eq_sommerfeld_intro}
\frac{\partial u}{\partial |\bx|}(\bx) - ik u(\bx)
=
o\left( |\bx|^{(-d +1)/2}\right)
\text{ as }
|\bx| \to +\infty.
\end{equation}
\end{subequations}
A particularly important scenario covered by \eqref{eq_helmholtz_intro} is the scattering of a
plane-wave, where the right-hand side takes the form
\begin{equation}
\label{eq_rhs_scattering}
f \eq \left (k^2 \mu + \div \left (\BA \grad \cdot\right )\right ) e^{ik\bd\cdot \bx},
\end{equation}
where $\bd \in \R^d$, $|\bd| = 1$ is the incidence direction
(such right-hand sides are compactly supported due to the assumptions on $\mu$ and $\BA$).

As we propose a ``volumic'' method, we will actually replace the Sommerfeld radiation
condition \eqref{eq_sommerfeld_intro} by a Perfectly Matched Layer (PML). This approach is
entirely standard \cite{berenger_1994,collino_monk_1998a,galkowski2021perfectly}, and amounts
to slightly modifying the coefficients $\mu$ and $\BA$ in \eqref{eq_volume_intro}.
This process is detailed in Section \ref{section_model_problem}.
We therefore do not directly discretize \eqref{eq_helmholtz_intro}, but rather,
an alternative version where the coefficients have been modified away from the origin
in \eqref{eq_volume_intro} and the Sommerfeld radiation condition \eqref{eq_sommerfeld_intro}
replaced by the requirement that the solution sits in $H^1(\R^d)$. Crucially, the solutions to
both problem coincide close to the origin. In what follows, we work with modified coefficients
$\mu$ and $\BA$ that incorporate a PML, but keep the same notation for simplicity.

In this work, we investigate the use of discretization spaces based on Gaussian coherent
states (GCS), that is, functions of the form
\begin{equation*}
\Phi_{k,\bx_0,\bxi_0}(\bx)
\eq
\left (\frac{k}{\pi}\right )^{d/4} e^{-\frac{k}{2}|\bx-\bx_0|^2} e^{-ik\bxi_0 \cdot (\bx-\bx_0)},
\end{equation*}
where $\bx_0,\bxi_0 \in \R^d$ are user-selected parameters.
The idea of decomposing a function as a discrete sum of Gaussian
coherent states goes back to \cite{Gabor}.  Here, following
\cite{chaumontfrelet_ingremeau_2022a,daubechies_grossman_meyer_1986a},
we focus on a lattice of phase-space points
$[\xxkm,\xikn] \eq \sqrt{(kR)^{-1}\pi} [\bm,\bn]$, $[\bm,\bn] \in \Z^{2d}$,
that are spaced by $\sim (kR)^{-1/2}$, and we write
\begin{equation*}
\gskmn \eq \Phi_{k,\xxkm,\xikn}.
\end{equation*}
Then, our a discretization space is of the form
\begin{equation*}
W_\Lambda
\eq
\mathrm{Vect} \left\{ \gskmn; \;\; [\bm,\bn] \in \Lambda \right\},
\end{equation*}
where $\Lambda \subset \Z^{2d}$ is a carefully chosen set of indices.

For the sake of simplicity, we will assume in the introduction that the domain is non-trapping.
Our first result is that if $\Lambda$ is chosen as
\begin{equation*}
\Lambda \eq \left \{
[\bm,\bn] \in \Z^{2d};
\;\;
|\xxkm|^2 + |\xikn|^2 \leq \rho kR
\right \},
\end{equation*}
for $\rho > 0$, we have
\begin{equation}
\label{eq_approximation_fixed}
\dim W_\Lambda \simeq \rho^d (kR)^d
\;\; \text{ and } \;\;
k^2 \min_{w \in W_\Lambda} \|u-w\|_{\HH^1(\R^d)}
\leq
C \rho^{-1/2} \|f\|_{L^2(\R^d)},
\end{equation}
for a general $f \in L^2(\R^d)$ with $\operatorname{supp} f \subset B(0,R)$,
where $\|{\cdot}\|_{\HH^1(\R^d)}$ is a $H^1(\R^d)$-norm including a weight at infinity
(see \eqref{eq_weighted_norm_k} below). As we describe
in more length afterwards, this result is not very impressive on its own. Specifically,
it is a standard approximation result similar to polynomial approximations: we need a fixed
number of points per wavelength to achieve a constant accuracy. Our second result, which is key,
deals with the case where $f$ takes the particular form \eqref{eq_rhs_scattering}. In this case,
we select
\begin{equation*}
\Lambda
\eq
\left \{
[\bm,\bn] \in \Z^{2d};
\;\;
|p(\xxkm,\xikn)|
\leq
(kR)^{-1/2+\varepsilon}
\right \}
\end{equation*}
where $\varepsilon > 0$ can be selected arbitrarily small and
\begin{equation}\label{e:defp}
p(\bx,\bxi) \eq \BA(\bx) \bxi \cdot \bxi -\mu(\bx) \qquad \forall \bx,\bxi \in \R^d,
\end{equation}
is the principal symbol associated with the differential operator in \eqref{eq_helmholtz_intro},
where the coefficients include a PML.
We then have
\begin{equation}
\label{eq_approximation_asymptotic}
\dim W_\Lambda \simeq (kR)^{d-1/2+\varepsilon}
\;\; \text{ and } \;\;
\min_{w \in W_\Lambda} \|u-w\|_{\HH^1(\R^d)} \leq C_{\varepsilon,m} (kR)^{-m}
\qquad
\forall m \in \N.
\end{equation}
It means that for right-hand sides corresponding to scattering problems (and actually,
a wider family of right-hand sides), Gaussian coherent states provide an accurate
solution with $\CO((kR)^{d-1/2+\varepsilon})$ DOFs. In fact, the convergence is even
super-algebraic as the frequency increases.

To put \eqref{eq_approximation_fixed} and \eqref{eq_approximation_asymptotic}
into perspective, we compare them with other standard methods. Actually,
there are several options to numerically solve \eqref{eq_helmholtz_intro} (either with
the Sommerfeld condition \eqref{eq_sommerfeld_intro} or with a PML approximation), that
we review below.

The most versatile approach is probably the finite element method (FEM). The method hinges
on a triangulation of the domain into elements of size $h$, and piecewise polynomial basis
functions of degree $p$. It can be shown that if $p$ grows logarithmically with $k$, then
the condition that $kh/p$ is constant provides (at least) a constant accuracy as $k$ increases
\cite{lafontaine_spence_wunsch_2022a,melenk_sauter_2010a,melenk_sauter_2011a}.
As a result, high-order FEM essentially requires $\mathcal O((kR)^d)$ degrees of freedom (DOFs)
to achieve a constant accuracy. The resulting matrix is sparse.

Trefftz-like methods are similar to FEM in that they also rely on a mesh of the domain,
but the polynomial shape functions are replaced by local solutions to the Helmholtz problem,
such as plane-waves \cite{hiptmair_moiola_perugia_2016a}, or generalised plane-waves
\cite{imbert2014generalized,imbert2021amplitude}. There are many ways to ``glue''
these local solution together, including partition of unity methods
\cite{melenk1996partition},
least squares methods \cite{monk1999least}, the
ultra weak variational method \cite{cessenat2003using},
the discontinuous enrichment method \cite{farhat2001discontinuous}
or the variational theory of complex rays \cite{riou2008multiscale}.
While these methods typically induce a large reduction of the number of DOFs as
compared to FEM, they usually still need at least $\mathcal O((kR)^d)$ DOFs,
see, e.g., 
\cite{chaumontfrelet_valentin_2020a,gittelson_hiptmair_perugia_2009a,hiptmair_moiola_perugia_2011a}.
Similar to FEM, the resulting matrix is sparse.

The next family of methods we want to mention are boundary element methods (BEM)
\cite{sauter_schwab_2010a}. These methods rely on boundary integral equations which,
strictly speaking, are not available for smoothly varying coefficients, since the expression
of Green's function must be available. It is nevertheless interesting to include them in the
comparison. These methods typically provide a constant accuracy with only
$\mathcal O((kR)^{d-1})$ DOFs \cite{galkowski_spence_2022a}.
However, the resulting matrix is dense and its entries are costly to compute.
These issues can be mitigated using compression techniques,
such as the fast multi-pole method \cite{Greengard:1987:FMM} or
hierarchical matrices \cite{Hackbusch:2015:HMM}.
 
Finally, asymptotic methods rely on the fact that when the frequency is very large,
it is sometimes possible to simplify the search for the solution of the Helmholtz
equation and the properties of the solution itself to computations involving only the underlying
classical dynamics \cite{engquist_runborg_2003a}. This is done using tools of semi-classical
analysis, such as the WKB method and can lead to discrete problems with a number of DOFs
independent of $k$.
The main drawback of these approaches is that they are only asymptotically valid: they do
not converge for fixed value of $k$. Besides, it is not always clear from which range of $k$
they are relevant.

As compared to FEM, the proposed GCS method thus gains ``half a dimension'' at high-frequencies,
but it is still half a dimension higher than BEM. As compared to BEM however, our methodology
has the advantage to apply in a generic framework where the Green's function is not available.
Another important comment is that in the (very) high-frequency regime, our method is more
expensive than asymptotic methods. However, asymptotic methods cannot converge at fixed $k$,
which our method does. This is summarized in Table \ref{table_costs}.

\begin{table}
\centering
\begin{tabular}{|c|cccc|}
\hline
                         & FEM      & GCS          & BEM          & asymptotic \\
\hline
Cost                     & $(kR)^d$ & $(kR)^{d-1/2}$ & $(kR)^{d-1}$ & $1$        \\
\hline
Heterogeneous media       & yes      & yes            & no           & yes        \\
\hline
Convergence at fixed $k$ & yes & yes & yes & no \\
\hline
\end{tabular}
\caption{Comparison of commonly used discretization techniques}
\label{table_costs}
\end{table}

In addition to the approximability results \eqref{eq_approximation_fixed} and
\eqref{eq_approximation_asymptotic}, we also present a least-squares method
based on Gaussian coherent states for Problem \eqref{eq_helmholtz_intro}. As we show,
the convergence of the method is easily established. Besides, although the matrix
is dense, we show that the entries decay super-algebraically away from the diagonal.
As a result, the matrix is essentially banded, and we believe that efficient linear
solvers can be devised. This will be analysed in more depth in future works.



We finally present a set of one-dimensional numerical experiments using the proposed
least-squares method. Although the setting is rather simple, the results perfectly fit
the theoretical analysis and readily shows that proposed approach allows for a drastic
reduction of the number of DOFs in the high-frequency regime.

To the best of our knowledge, our micro-locally adapted spaces of Gaussian coherent states
appear to be entirely original, but we would like to mention that similar basis functions
have already been employed to discretize PDE problems. In particular, generalised coherent states
like Hagedorn wavepackets were used to describe the solution of time-dependent Schr\"odinger
equation in \cite{Faou:2009:CSQ,Gradinaru:2014:CSW,Gradinaru:2021:HWS,lasser2020computing}.

We would like to emphasize that the main goal of our work is to understand how
many DOFs are required to maintain a constant accuracy as $k$ increases, rather than
deriving precise convergence rates in terms of DOFs for a fixed frequency. For the
sake of completeness, we nevertheless provide a quick comparison here. Assuming
that $k$ is fixed and that the right-hand side $f$ is smooth, the results in
\cite{chaumontfrelet_ingremeau_2022a} (see also \cite[Chapter 11]{Gro} and
Corollary \ref{corollary_least_squares_fixed} below) show that the convergence
to the solution will be super-algebraic in terms of the number of DOFs. Without
further assumption on $f$, such rates are equivalent to $hp$ finite element
and Trefftz methods. Under the additional assumption that the solution $u$ is
(piecewise) analytic, $hp$ finite element and Trefftz methods can achieve an
exponential convergence rate in terms of DOFs \cite{schwab_1998a}, with an improved
rate for Trefftz methods as compared to finite element methods
\cite{hiptmair_moiola_perugia_2011a,melenk_1999a}. To the best our knowledge,
it is an open question whether such a result holds true for the present method.

The remainder of our work is organised as follows. In Section \ref{sec:setting},
we precise the setting and state our key approximation result \eqref{eq_approximation_asymptotic}
in its most general form. Section \ref{sec:proofs} contains the proof of our findings.
In Section \ref{sec:helmholtz}, we apply the general theory of Sections \ref{sec:setting}
and \ref{sec:proofs} to our scattering model problem with PML. Numerical examples are reported
in Section \ref{sec:numerics}. Finally, Appendix \ref{sec:GS} collects technical results
concerning Gaussian coherent states.

\section{Setting and main results}
\label{sec:setting}

\subsection{Notations}
\label{section_h_notations}

Throughout this work $\hbar \in \LH \subset (0,1]$ will denote a small parameter.
When applying our general results to the Helmholtz equation, we will have $\hbar \sim (kR)^{-1}$,
so that considering the set $(0,1]$ amounts to ignoring low frequencies, and focusing on
high frequencies when $\hbar \to 0$. For the sake of generality, we restrict our analysis
to a subset $\LH \subset (0,1]$ for reasons that will become apparent in Section
\ref{sec:helmholtz}. Notice that the case $\LH = (0,1]$ is not excluded.

\subsubsection{Basic notation}

The canonical basis of $\R^d$ or of $\C^d$ will be denoted by $(\boldsymbol{e}_1,..., \boldsymbol{e}_d)$.
If $\bx,\by \in \C^d$, we write
\begin{equation*}
\bx \cdot \by  \eq \sum_{j=1}^d x_j y_j
\end{equation*}
without complex conjugation on the second argument, and
$|\bx| = (\bx \cdot \overline{\bx})^{1/2}$
is the usual Euclidean norm.

For a multi-index $\balpha \in \N^d$, $[\balpha] \eq \alpha_1+\dots+\alpha_d$
denotes its usual $\ell_1$ norm. If $v: \R^d \to \C$, the notation
\begin{equation*}
\partial^{\balpha} v
\eq
\frac{\partial^{\alpha_1}}{\partial x_1}
\dots
\frac{\partial^{\alpha_d}}{\partial x_d} v
\end{equation*}
is employed for the partial derivatives in the sense of distributions, whereas
$\bx^{\balpha} \eq x_1^{\alpha_1} \cdot \ldots \cdot x_d^{\alpha_d}$.
Finally, if $\bbeta \in \N^d$ is another multi-index, we will sometimes need the notation
\begin{equation*}
\left (
\begin{array}{c}
\balpha \\ \bbeta
\end{array}
\right )
=
\prod_{j=1}^d
\left (
\begin{array}{c}
\alpha_j \\ \beta_j
\end{array}
\right ),
\end{equation*}
and the notation $\balpha \leq \bbeta$ means that $\alpha_j \leq \beta_j$
for all $j \in \{1,\dots,d\}$. 

If $\bn \in \Z^d$, we employ the notation $|\bn|^2 \eq n_1^2 + \dots + n_d^2$
for its $\ell_2$ norm. Finally, if $\Lambda \subset \Z^{2d}$, $\ell^2(\Lambda)$
has its usual definition, and we denote by $\|\cdot\|_{\ell^2(\Lambda)}$ its usual norm.

\subsubsection{Key functional spaces}

In what follows, $L^2(\R^d)$ is the usual Lebesgue space of complex-valued square integrable
functions over $\R^d$. The usual norm and inner products of $L^2(\R^d)$ are respectively
denoted by $\|\cdot\|_{L^2(\R^d)}$ and $(\cdot,\cdot)$.

Since we are dealing with the (unbounded) $\R^d$ space, following
\cite{chaumontfrelet_ingremeau_2022a}, our analysis will require the
weighted Sobolev spaces
\begin{equation*}
\HH^p(\R^d) \eq \left \{
v \in L^2(\R^d) \; | \;
\bx^{\balpha} \partial^{\bbeta} v \in L^2(\R^d)
\quad \forall \balpha,\bbeta \in \N^d; \; [\balpha],[\bbeta] \leq p
\right \},
\end{equation*}
that we equip with the family of equivalent $\hbar$-weighted norms given by
\begin{equation*}
\|v\|_{\HH^p_\hbar(\R^d)}^2
\eq
\sum_{[\balpha] \leq p} \sum_{q \leq p- [\balpha]}
\hbar^{2[\balpha]} \||\bx|^q \partial^{\balpha} v\|_{L^2(\R^d)}^2
\end{equation*}
for all $p \in \N$.

$C^0(\R^d)$ is the set of complex-valued continuous functions defined over $\R^d$,
and $C^\ell(\R^d)$ is the set of functions $v: \R^d \to \C$ such that
$\partial^{\balpha} v \in C^0(\R^d)$ for all $\balpha \in \N^d$ with $[\balpha] \leq \ell$.
We introduce the notation
\begin{equation*}
\|v\|_{C^\ell(\R^d)}
\eq
\max_{\substack{\balpha \in \N^d \\ [\balpha] \leq \ell}}
\max_{\bx \in \R^d} |(\partial^{\balpha} v)(\bx)|
\qquad \forall v \in C^\ell(\R^d)
\end{equation*}
and $C^\ell_{\rm b}(\R^d)$ is the subset of functions $v \in C^\ell(\R^d)$
such that $\|v\|_{C^\ell(\R^d)} < +\infty$. We also set
\begin{equation*}
C^\infty_{\rm b}(\R^d) \eq \bigcap_{\ell \in \N} C^\ell_{\rm b}(\R^d).
\end{equation*}
Finally, if $\Omega \subset \R^d$ is an open set, we denote by $C_c^\infty(\Omega)$
the set of smooth functions whose support is a compact subset of $\Omega$.

\subsection{The frame of Gaussian coherent states}

The goal of this work is to efficiently approximate the solution $u_\hbar$
to the equation $P_\hbar u_\hbar = f$ with a finite span of Gaussian coherent
states. For $[\bm,\bn] \in \Z^{2d}$, we thus consider the Gaussian state
\begin{equation*}
\gshmn(\bx)
\eq
(\pi\hbar)^{-d/4}
e^{-\frac{1}{2\hbar}|\bx-\xxhm|^2} e^{\frac{i}{\hbar}\xihn \cdot (\bx-\xxhm)},
\end{equation*}
where $\xxhm \eq \sqrt{\pi\hbar} \bm$ and $\xihn \eq \sqrt{\pi\hbar} \bn$.
The family of Gaussian coherent states $(\gshmn)_{[\bm,\bn] \in \Z^{2d}}$
actually forms a \emph{frame} over $L^2(\R^d)$, meaning
there exists two constants $0 < \alpha < \beta < +\infty$ solely
depending on $d$ such that
\begin{equation*}
\alpha \|v\|_{L^2(\R^d)}^2
\leq
\sum_{[\bm,\bn] \in \Z^{2d}} |(v,\gshmn)|^2
\leq
\beta \|v\|_{L^2(\R^d)}^2 \qquad \forall v \in L^2(\R^d).
\end{equation*}
This result was first proved in \cite{daubechies_grossman_meyer_1986a}, but the idea of decomposing a function as a discrete sum of Gaussian states goes back to \cite{Gabor},
where it was proved that the span of $(\gshmn)_{[\bm,\bn] \in \Z^{2d}}$ is dense in $L^2(\R^d)$.

Actually, the frame property implies that there exists another family of functions
$(\gshmn^\star)_{[\bm,\bn] \in \Z^{2d}}$ called the dual frame such that
\begin{equation}
\label{eq_dual_frame}
v
=
\sum_{[\bm,\bn] \in \Z^{2d}} (v,\gshmn^\star) \gshmn
\end{equation}
for all $v \in L^2(\R^d)$.
It is thus clear that any $v \in L^2(\R^d)$ may be well-approximated by (a large number of)
Gaussian states. As we are going to develop hereafter, when considering the solution to
a high-frequency PDE problem, a good approximation may be obtained with few
Gaussian states, by carefully selecting the indices $[\bm,\bn]$ in \eqref{eq_dual_frame}.

\begin{remark}[General expansions in the Gaussian frame]
\label{rem:BoundedUtile}
The family $(\gshmn)$ is not a Riesz basis,  so that the expansion \eqref{eq_dual_frame} of $v$
as a sum of $\gshmn$ is not unique.  However,  a crucial property of \eqref{eq_dual_frame} is
that this expansion is stable, in the sense that 
\begin{equation*}
\sum_{[\bm,\bn] \in \Z^{2d}} |(v,\gshmn^\star)|^2 \leq \gamma \|v\|_{L^2(\R^d)}^2,
\end{equation*}
where $\gamma$ only depends on $d$, the dual frame being itself a frame. This is especially
important at the numerical level in the presence of round-off errors
\cite{Adcock:2019:FNA}.
\end{remark}

\subsection{Settings and key assumptions}
\label{ssec:general}

Throughout this work, we consider a second order differential operator on
$\R^d$ depending on $\hbar$, and taking the form
\begin{equation}
\label{eq:FormeGeneraleOperateur}
(P_{\hbar} v)(\bx)
=
\hbar^2
\sum_{j,\ell=1}^d
a_{j\ell}^{\hbar}(\bx)
\frac{\partial^2 v}{\partial x_j \partial x_\ell}(\bx)
+
i\hbar\sum_{j=1}^d b_{j}^\hbar(\bx) \frac{\partial v}{\partial  x_j }(\bx)
+
c^\hbar(\bx) v(\bx),
\end{equation}
where $a_{j\ell}^\hbar,b_j^\hbar,c^\hbar \in \CIB(\R^d)$ for $1 \leq j,\ell \leq d$.
For the sake of simplicity, we introduce
\begin{equation*}
C_{{\rm coef},p}
\eq
\sup_{\hbar \in \LH}
\left (
\sum_{j,\ell=1}^d \|a_{j,\ell}^\hbar\|_{C^p(\R^d)}
+
\sum_{j=1}^d \|b_j^\hbar\|_{C^p(\R^d)}
+
\|c^\hbar\|_{C^p(\R^d)}
\right )
\qquad
\forall p \in \N,
\end{equation*}
and assume that $C_{{\rm coef},p} < +\infty$ for all $p \in \N$.

The principal symbol of $P_\hbar$ is the function $p_\hbar \in C^\infty(\R^{2d})$ defined by
\begin{equation}
\label{eq_symbol}
p_\hbar(\bx,\bxi)
\eq
\sum_{j,\ell=1}^d
a_{j,\ell}^\hbar(\bx) \xi_i \xi_j
+
\sum_{j=1}^d b_j^\hbar(\bx) \xi_j
+
c^\hbar(\bx).
\end{equation}
For the sake of shortness, we will often write
$\mathrm{p}_\hbar (\bm,\bn) \eq p_\hbar(\xxhm,\xihn)$
for $[\bm,\bn] \in \Z^{2d}$.

\begin{remark}[$\hbar$-dependent symbol]
When considering the standard Helmholtz differential operator
$P_\hbar v \eq -v -\hbar^2 \Delta v$ the symbol simply
reads $p(\bx,\bxi) = 1-|\bxi|^2$ and in particular,
it does not depend on $\hbar$. It is however interesting
to allow for a mild dependence on $\hbar$ in the symbol.
This is for instance the case when considering Helmholtz
problems in the presence of dissipative materials. In this case,
the differential operator reads $P_\hbar v \eq -v -i\gamma \hbar v-\hbar^2 \Delta v$,
where the function $\gamma \geq 0$ models the dissipation. One readily
see in this case that the symbol is $p_\hbar(\bx,\bxi) = 1-i\gamma(\bx)\hbar-|\bxi|^2$.
\end{remark}

\begin{remark}[non-divergence form]
In contrast to \eqref{eq_helmholtz_intro} in the introduction, we present our model PDE
in non-divergence form in \eqref{eq:FormeGeneraleOperateur}. The key reason behind this
choice is that it makes the link between the PDE operator and its symbol simpler. Notice
that since we are only considering smooth coefficients in this work, this choice is absolutely
not restrictive, and \eqref{eq_helmholtz_intro} can be easily recast into
\eqref{eq:FormeGeneraleOperateur}.
\end{remark}

Along with the smoothness of the coefficients, we make two key assumptions.
First, we assume that $P_\hbar: \HH^p(\R^d) \to \HH^p(\R^d)$ is invertible
with the norm of $P_\hbar^{-1}$ being polynomially bounded in $\hbar^{-1}$.
Specifically, we assume that for all $f \in L^2(\R^d)$ there exists a unique
$u_\hbar \in L^2(\R^d)$ such that $P_\hbar u_\hbar = f$. In addition, there exists $N \in \N$
such that for all $p \in \N$, if $f \in \HH^p(\R^d)$, then $u_\hbar \in \HH^p(\R^d)$ with
\begin{equation}
\label{eq_polynomial_resolvant}
\|u_\hbar\|_{\HH_\hbar^p(\R^d)}
\leq
C_{{\rm sol},p} \hbar^{-N} \|f\|_{\HH_\hbar^p(\R^d)}
\quad
\forall \hbar \in \LH
\end{equation}
for some constant $C_{{\rm sol},p}$ independent of $\hbar$.

In the context of high-frequency scattering, these assumptions are reasonable and
hold in a variety of situations. We refer the reader to Remark \ref{rem:ResolEst}
in Section \ref{section_model_problem} below where we expand on that aspect.

Our second assumption is that there exists a value $\delta_0 > 0$ such that
\begin{equation}
\label{eq_assumption_symbol_bounded}
\exists D_0>0 \text{ such that } \forall \hbar \in \LH, ~~ 
\{(\bx,\bxi)\in \R^{2d} ; \; |p_\hbar(\bx,\bxi)| < \delta_0\} \subset B(0, D_0).
\end{equation}

In the remainder of this work, we allow generic constants $C$
to depend on $\{C_{{\rm coef},p}\}_{p \in \N}$, $\{C_{{\rm sol},p}\}_{p \in \N}$,
$N$ and $D_0$. We also employ the notation $C_{\alpha,\beta,\dots}$ if the constant
$C$ is additionally allowed to depend on other previously introduced quantity
$\alpha,\beta,\dots$

\subsection{Statement of the approximability result}
\label{ssec:statement}

Our main result is that, if $f$ is micro-localised near the set
$\{ (\bx, \bxi) \in \R^{2d}; p_\hbar(\bx, \bxi)=0\}$, then so is the solution
$u_\hbar$ to $P_\hbar u_\hbar = f$. This is a standard result when micro-localisation
is understood in terms of pseudo-differential operators (see for instance
\cite[Theorem 6.4]{zworski2012semiclassical}), but here,  by micro-localisation properties,
we mean that $f$ can be approached by a linear combination of $\gshmn$ with $[\bm,\bn]$ in
a certain region of $\Z^{2d}$. Hence, our results may not be easily recovered from
standard results in semiclassical analysis (see e.g. \cite[Theorem 6.4]{zworski2012semiclassical}).

\begin{theorem}[Approximability for Gaussian state right-hand sides]
\label{theorem_approximability}
Let $0 < \varepsilon < 1/2$ and $0 \leq \alpha < \delta_0/2$.
For $\hbar \in \LH$, consider the right-hand side
\begin{equation*}
f_\hbar \eq \sum_{[\bm,\bn] \in \Lambda_{\hbar,{\rm rhs}}} F_{\bm,\bn}^\hbar \gshmn.
\end{equation*}
where $F^\hbar \in \Lambda_{\hbar,\rm rhs}$ with 
\begin{equation}\label{eq:FormeRHS}
\Lambda_{\hbar,{\rm rhs}}
\eq
\left \{
[\bm,\bn] \in \Z^{2d}
\; | \;
|p_\hbar(\xxhm,\xihn)| \leq \alpha + 2\hbar^{1/2-\varepsilon}
\right \}.
\end{equation}
Then, if $u_\hbar$ is the solution to $P_\hbar u_\hbar = f_\hbar$, we have
\begin{equation*}
\left \|
u_\hbar
-
\sum_{[\bm, \bn] \in \Lambda_{\hbar,{\rm sol}}} (u_\hbar,\gshmn^\star) \gshmn
\right \|_{\HH_\hbar^p(\R^d)}
\leq
C_{\varepsilon,p,m} \hbar^m \|F_\hbar\|_{\ell^2(\Z^{2d})}
\quad
\forall m \in \N.
\end{equation*}
for all $p \in \N$ with
\begin{equation*}
\Lambda_{\hbar,{\rm sol}}
\eq
\left \{
[\bm,\bn] \in \Z^{2d}
\; | \;
|p_\hbar(\xxhm,\xihn)| \leq \alpha + 4\hbar^{1/2-\varepsilon}
\right \}.
\end{equation*}
\end{theorem}

Notice that the index sets $\Lambda_{\hbar,{\rm rhs}}$ and $\Lambda_{\hbar,{\rm sol}}$
depend of $\varepsilon$, which explains why the constants in Theorem \ref{theorem_approximability}
and Corollary \ref{corollary_approximability} depend on $\varepsilon$.

In practice, the right-hand side of the problem is not a finite linear combination
of Gaussian coherent states. However, many right-hand sides of interest become
well-approximated by such combination in the high-frequency regime. This is
in particular the case when considering scattering by a plane-wave (see Lemma
\ref{Lem:ApproxPW} below).

\begin{corollary}[Approximability for micro-localised right-hand sides]
\label{corollary_approximability}
Let $p \geq 0$. Consider a set of right-hand sides
$(f_\hbar)_{\hbar \in \LH} \subset \HH^p(\R^d)$
and assume that there exists a set of sequences
$(F^{\hbar})_{\hbar \in \LH} \subset \ell^2(\Z^{2d})$
such that
\begin{align*}
\|F^{\hbar}\|_{\ell^2(\Z^{2d})} &\leq C \\
\|f_\hbar - \sum_{[\bm,\bn] \in \Lambda_{\hbar,{\rm rhs}}} F^{\hbar}_{\bm,\bn}\gshmn\|_{\HH_\hbar^p(\R^d)}
&\leq
C_{\varepsilon,m} \hbar^m \quad \forall m \in \N
\end{align*}
for all $\hbar \in \LH$. Then, we have
\begin{equation*}
\left \|
u_\hbar
-
\sum_{[\bm,\bn] \in \Lambda_{\hbar,{\rm sol}}} (u_{\hbar},\gshmn^\star)\gskmn
\right \|_{\HH_\hbar^{p+2}(\R^d)}
\leq
C_{\varepsilon,p,m}\hbar^m
\quad \forall m \in \N.
\end{equation*}
\end{corollary}

\section{Proof of Theorem \ref{theorem_approximability}}
\label{sec:proofs}

This section is devoted to the detailed proofs of Theorem \ref{theorem_approximability}
and Corollary \ref{corollary_approximability}.

\subsection{Preliminary results on Gaussian states}
\label{sec:properties}

We start by stating that the following bound
\begin{equation}
\label{eq_norm_gskmn}
\|\gshmn\|_{\HH_\hbar^s(\R^d)}
\leq
C_s (1+ (\hbar^{1/2}|[\bm,\bn]|)^s)
\end{equation}
holds true for all $s \in \N$, see \cite[Lemma C.1]{chaumontfrelet_ingremeau_2022a}.
We will also need the following expansion result.

\begin{proposition}[Tight expansion]
\label{proposition_tight_expansion}
For all $[\bm,\bn] \in \Z^{2d}$, there exists a sequence of coefficients
$U^{\bm,\bn} \subset \ell^1(\Z^{2d})$ such that
\begin{equation*}
\|U^{\bm,\bn}\|_{\ell^p(\Z^{2d})} \leq C_p \quad \forall p \in [1,+\infty],
\end{equation*}
and for all $\varepsilon \in (0,1/2)$ and $s \in \N$, we have
\begin{equation}
\label{eq_tight_expansion}
\left \|
\gshmn^\star-
\sum_{\substack{[\bm',\bn'] \in \Z^{2d} \\ |[\bm,\bn]-[\bm',\bn']| \leq \hbar^{-\varepsilon}}}
U^{\bm,\bn}_{\bm',\bn'} \gshmnp
\right \|_{\HH_\hbar^s(\R^d)}
\leq
C_{\varepsilon,s,m} \hbar^m
\quad
\forall m \in \N.
\end{equation}
\end{proposition}

\begin{proof}
We start by applying \eqref{eq_dual_frame} to $v = \gshmn^\star$, leading to
\begin{equation*}
\gshmn^\star = \sum_{[\bm',\bn'] \in \Z^{2d}} (\gshmn^\star,\gshmnp^\star) \gskmnp.
\end{equation*}
Next, we recall from \cite[Proposition 4.2]{chaumontfrelet_ingremeau_2022a} that
\begin{equation*}
|(\gshmn^\star,\gshmnp^\star)|
\leq
C e^{-|[\bm,\bn]-[\bm',\bn']|^{1/2}}.
\end{equation*}
As a result, we define $U^{\bm,\bn}_{\bm',\bn'} \eq (\gshmn^\star,\gshmnp^\star)$, so that
$U^{\bm,\bn}$ indeed belongs to $\ell^p(\Z^{2d})$ for $1 \leq p \leq +\infty$, and
\begin{equation*}
\CE
\eq
\gshmn^\star-
\sum_{\substack{[\bm',\bn'] \in \Z^{2d} \\ |[\bm,\bn]-[\bm',\bn']| \leq \hbar^{-\varepsilon}}}
U^{\bm,\bn}_{\bm',\bn'} \gshmnp
=
\sum_{\substack{[\bm',\bn'] \in \Z^{2d} \\ |[\bm,\bn]-[\bm',\bn']| > \hbar^{-\varepsilon}}}
U^{\bm,\bn}_{\bm',\bn'} \gshmnp.
\end{equation*}
We then observe that
\begin{align*}
\|\CE\|_{\HH_\hbar^s(\R^d)}
&\leq
\sum_{\substack{[\bm',\bn'] \in \Z^{2d} \\ |[\bm,\bn]-[\bm',\bn']| > \hbar^{-\varepsilon}}}
|U^{\bm,\bn}_{\bm',\bn'}| \cdot \|\gshmnp\|_{\HH_\hbar^s(\R^d)}
\\
&\leq
C 
\sum_{\substack{[\bm',\bn'] \in \Z^{2d} \\ |[\bm,\bn]-[\bm',\bn']| > \hbar^{-\varepsilon}}}
(1+|[\bm',\bn']|)^s e^{-|[\bm,\bn]-[\bm',\bn']|^{1/2}},
\end{align*}
due to \eqref{eq_norm_gskmn} and the fact that $\hbar \leq 1$. Estimate \eqref{eq_tight_expansion}
follows since
\begin{equation*}
\sum_{\substack{[\bm',\bn'] \in \Z^{2d} \\ |[\bm,\bn]-[\bm',\bn']| > \hbar^{-\varepsilon}}}
(1+ |[\bm',\bn']|)^s e^{-|[\bm,\bn]-[\bm',\bn']|^{1/2}}
\leq
C_{\varepsilon,s,m} \hbar^m.
\hfill \qed
\end{equation*}
\end{proof}

We close this section with two technical results. As we believe they are of independent interest,
and because their proof require tedious computations, they are presented later in Appendix
\ref{sec:GS}.

\begin{proposition}[Quasi orthogonality]
\label{proposition_quasi_orthogonality}
Consider two sets of indices $\Lambda,\Lambda' \subset \Z^{2d}$ with
\begin{equation*}
\rho \eq \dist(\Lambda,\Lambda') > 0,
\end{equation*}
and such that there exists $\mu > 0$ with
\begin{equation*}
\Lambda\subset B(0,\mu).
\end{equation*}
Consider $L \in \N$, smooth coefficients
$(\TA_{\ba})_{\ba \in \N^d} \subset C^\infty_{\rm b}(\R^d)$
and the differential operator
\begin{equation*}
\CP_{\hbar,L,\TA}
\eq
\sum_{\substack{\balpha \in \N^d \\ [\balpha] \leq L}}
\hbar^{[\balpha]} \TA_{\balpha} \partial^{\balpha}.
\end{equation*}
Then, for all $q > 0$ and $\balpha \in \N^d$, we have
\begin{equation}
\begin{aligned}
\label{eq_phase_space_localisation}
\sum_{[\bm,\bn] \in \Lambda} \sum_{[\bm',\bn'] \in \Lambda'}
|(\bx^{\balpha}\CP_{\hbar,L,\TA} \gshmn,\gshmnp)|^q
& \leq
C_{L,\TA,q,m,\balpha} (1 + (\mu \hbar^{-1/2})^{([\balpha]+L)q}) |\Lambda| (1+\rho)^{-m}\\
& \leq
C'_{L,\TA,q,m,\balpha} (1 + (\mu \hbar^{-1/2})^{([\balpha]+L)q + 2d})(1+\rho)^{-m}
\end{aligned}
\end{equation}
for all $m \in \N$.
\end{proposition}

\begin{proposition}[Control of $(P_\hbar-p_\hbar)^L$]\label{prop:OpMoinsSymb}
For all $[\bm,\bn] \in \Z^{2d}$, we have
\begin{subequations}
\label{eq_iterated_operators}
\begin{equation}
\label{eq_iterated_operator}
\left \|
\left (P_\hbar - p_\hbar(\xxkm,\xikn) \right )^L \gshmn
\right \|_{L^2(\R^d)}
\leq
C_L (1+ \hbar|\bn|^2)^{L/2} \hbar^{L/2},
\end{equation}
and
\begin{equation}
\label{eq_iterated_adjoint_operator}
\left \|
\left(P_\hbar^\star - \overline{p}_\hbar(\xxkm,\xikn) \right)^L \gshmn
\right \|_{L^2(\R^d)}
\leq
C_L (1+ \hbar|\bn|^2)^{L/2} \hbar^{L/2}.
\end{equation}
\end{subequations}
\end{proposition}

The first part of Proposition \ref{prop:OpMoinsSymb} directly follows from
Proposition \ref{Prop:OpMoinsSYmbApp}. In the second part, $P_\hbar^\star$
denotes the formal adjoint of $P_{\hbar}$, which is also of the form
\eqref{eq:secondorderop}.
A straightforward computation shows that its symbol $p_h^\star$ satisfies
$|p_h(\bx,\bxi) - p_h^\star(\bx,\bxi)|
\leq C\hbar^{1/2} |\bxi|^2$ uniformly in $\bx \in \R^d$,
so that the second part of Proposition \ref{prop:OpMoinsSymb} also follows from Proposition
\ref{Prop:OpMoinsSYmbApp}.

\subsection{Main proof}

We focus here on the proof of Theorem \ref{theorem_approximability}.
The notation
\begin{equation*}
\Lambda(a,b)
\eq
\left \{
[\bm,\bn] \in \Z^{2d} \; | \;
a \leq |p(\xxhm,\xihn)| \leq b
\right \},
\end{equation*}
where $0 \leq a \leq b \leq +\infty$, will be useful. In what follows,
we fix an $\varepsilon \in (0,1/2)$, and consider a right-hand side
micro-localised near $\{p_\hbar=0\}$. Specifically, we will assume that
\begin{equation*}
f_\hbar \eq \sum_{[\bm,\bn] \in \Lambda_{\hbar,{\rm rhs}}} F^{\hbar}_{\bm,\bn} \gshmn,
\end{equation*}
where $F^\hbar \in \ell^2(\Lambda_{\rm rhs})$ with
$\Lambda_{\hbar,{\rm rhs}} \eq \Lambda(0,\alpha+2\hbar^{1/2-\varepsilon})$.
Our goal is to show that the associated solution $u_\hbar$ is essentially
micro-localised near $\{p_\hbar=0\}$ as well. Specifically, setting
$\Lambda_{\hbar,{\rm near}} \eq \Lambda(0,\alpha+4\hbar^{1/2-\varepsilon})$,
our goal will be to show that
\begin{equation*}
u^{\rm near}_\hbar \eq \sum_{[\bm,\bn] \in \Lambda_{\hbar,{\rm near}}} (u_\hbar,\gshmn^\star) \gshmn
\end{equation*}
is ``close'' to $u_\hbar$.

The key idea is to separate the set of indices of $u_\hbar$ into $\Lambda_{\hbar,{\rm near}}$,
$\Lambda_{\hbar,{\rm mid}} \eq \Lambda(\alpha+4\hbar^{1/2-\varepsilon},\alpha+6\hbar^{1/2-\varepsilon})$ and
$\Lambda_{\hbar,{\rm far}} \eq \Lambda(\alpha+6\hbar^{1/2-\varepsilon},+\infty)$. We shall also need the ``enlarged''
sets
\begin{equation*}
\Lambda_{\hbar,{\rm near}}^\star
\eq
\Lambda(0,\alpha+5\hbar^{1/2-\varepsilon}),
\quad
\Lambda_{\hbar,{\rm mid}}^\star
\eq
\Lambda(\alpha+3\hbar^{1/2-\varepsilon},\alpha+7\hbar^{1/2-\varepsilon}),
\quad
\Lambda_{\hbar,{\rm far}}^\star
\eq
\Lambda(\alpha+5\hbar^{1/2-\varepsilon},+\infty),
\end{equation*}
for the test functions.

We first state some elementary properties of these sets of indices.
We do not report the (straightforward) proofs for the sake of shortness.

\begin{lemma}[Index sets]
\label{Lem:Index}
Assume that $\alpha+7\hbar^{1/2-\varepsilon} \leq \delta_0$,
then we have
\begin{equation}
\label{eq_dist_LBA}
\dist(\Lambda_{\hbar,{\rm far}},\Lambda_{\hbar,{\rm near}}^\star)
\geq
C \hbar^{-\varepsilon}
\quad
\dist(\Lambda_{\hbar,{\rm rhs}},\Lambda_{\hbar,{\rm mid}}^\star)
\geq
C \hbar^{-\varepsilon}
\quad
\dist(\Lambda_{\hbar,{\rm near}},\Lambda_{\hbar,{\rm far}}^\star)
\geq
C \hbar^{-\varepsilon}
\end{equation}
andthere exists $C>0$ such that
\begin{equation}
\label{eq_diam_LBA}
\Lambda_{\hbar,{\rm rhs}}^\star \cup \Lambda_{\hbar,{\rm near}}^\star \cup \Lambda_{\hbar,{\rm mid}}^\star \subset B(0,  C\hbar^{-1/2}).
\end{equation}
In addition, if $\hbar$ is small enough, we have
\begin{equation}
\label{eq_LBA_mid_star_inclusion}
\left \{
[\bm',\bn'] \in \Z^{2d} \; | \;
\exists [\bm,\bn] \in \Lambda_{\hbar,{\rm mid}};
\; |[\bm,\bn]-[\bm',\bn']| \leq \hbar^{-\varepsilon/2}
\right \}
\subset \Lambda_{\hbar,{\rm mid}}^\star.
\end{equation}
\end{lemma}

\begin{lemma}[Quasi orthogonality away from RHS micro-support]
\label{lemma_LBA_mid_star}
For $F^\hbar \in \ell^2(\Lambda_{\hbar,{\rm rhs}})$, consider the right-hand side
\begin{equation*}
f_\hbar = \sum_{[\bm,\bn] \in \Lambda_{\hbar,{\rm rhs}}} F^\hbar_{\bm,\bn} \gshmn
\end{equation*}
and the associated solution $u_\hbar$. Then, if $[\bm,\bn] \in \Lambda_{\hbar,{\rm mid}}^\star$, we have
\begin{equation*}
|(u_\hbar,\gshmn)|
\leq
C_{\varepsilon,m} \hbar^m \|F^\hbar\|_{\ell^2(\Z^{2d})}
\end{equation*}
for all $m \in \N$.
\end{lemma}

\begin{proof}
Throughout the proof, we fix a pair of indices $[\bm,\bn] \in \Lambda_{\hbar,{\rm mid}}^\star$.
By definition of $\Lambda_{\hbar,{\rm mid}}^\star$, the assumption that
$\alpha+7\hbar^{1/2-\varepsilon} \leq \delta_0$ and \eqref{eq_assumption_symbol_bounded},
we have
\begin{equation}
\label{tmp_bound_pmn_LBA_mid}
c\hbar^{1/2-\varepsilon} \leq |\pp_\hbar(\bm,\bn)| \leq C.
\end{equation}
In particular, we can write
\begin{equation*}
u_\hbar
=
\frac{1}{\pp_\hbar(\bm,\bn)} f_\hbar
+
\frac{1}{\pp_\hbar(\bm,\bn)} (P_\hbar-\pp_\hbar(\bm,\bn)) u_\hbar.
\end{equation*}
Since $f_\hbar$ is smooth, so is $u_\hbar$, and we can iterate this relation, leading to
\begin{equation*}
u_\hbar
=
\sum_{\ell=1}^r \frac{1}{\pp_\hbar(\bm,\bn)^\ell} (P_\hbar-\pp_\hbar(\bm,\bn))^{\ell-1} f_\hbar
+
\frac{1}{\pp_\hbar(\bm,\bn)^r} (P_\hbar-\pp_\hbar(\bm,\bn))^r u_\hbar,
\end{equation*}
and
\begin{align}
\nonumber
(u_\hbar,\gshmn)
&=
\sum_{\ell=1}^r \frac{1}{\pp_\hbar(\bm,\bn)^\ell}
(f_\hbar,(P_\hbar^\star-\overline{\pp}_\hbar(\bm,\bn))^{\ell-1} \gshmn)
\\
\label{tmp_iterated_relation}
&+
\frac{1}{\pp_\hbar(\bm,\bn)^r} (u_\hbar,(P_\hbar^\star-\overline{\pp}_\hbar(\bm,\bn))^r \gshmn)
\end{align}
for all $r \in \N$.

Then, if $[\bm',\bn'] \in \Lambda_{\hbar,{\rm rhs}}$, using the upper-bound
in \eqref{tmp_bound_pmn_LBA_mid}, we have
\begin{align*}
|(\gshmnp,(P_\hbar^\star-\overline{\pp}_\hbar(\bm,\bn))^{\ell-1} \gshmn)|
&=
\left |
\sum_{j=0}^{\ell-1}
\left (
\begin{array}{c}
\ell-1 \\ j
\end{array}
\right )
\overline{\pp}_\hbar(\bm,\bn)^{\ell-1-j} (\gshmnp,(P_\hbar^\star)^j \gshmn)
\right |
\\
&\leq
C_\ell 
\sum_{j=0}^{\ell-1} |(\gshmnp,(P_\hbar^\star)^j \gshmn)|.
\end{align*}
Recalling \eqref{eq_diam_LBA}, applying
\eqref{eq_phase_space_localisation} gives
\begin{equation*}
|(\gshmnp,(P_\hbar^\star-\overline{\pp}_\hbar(\bm,\bn))^{\ell-1} \gshmn)|
\leq
C_{\ell,n}(1+|[\bm,\bn]-[\bm',\bn']|)^{-n}
\leq
C_{\ell,n}\hbar^{\varepsilon n},
\end{equation*}
for all $n \in \N$, since $|[\bm,\bn]-[\bm',\bn']| \geq \hbar^{-\varepsilon}$
due to \eqref{eq_dist_LBA}.
The case $\ell=1$ also easily follows by Proposition \ref{proposition_quasi_orthogonality}
We then write that
\begin{align*}
|(f_\hbar,(P_\hbar^\star-\overline{\pp}_\hbar(\bm,\bn))^{\ell-1} \gshmn)|
&\leq
\sum_{[\bm',\bn'] \in \Lambda_{\hbar,{\rm rhs}}}
|F^\hbar_{\bm,\bn}||(\gshmnp,(P_\hbar^\star-\overline{\pp}_\hbar(\bm,\bn))^{\ell-1} \gshmn)|
\\
&\leq
C_{\ell,n} \hbar^{\varepsilon n}
\sum_{[\bm',\bn'] \in \Lambda_{\hbar,{\rm rhs}}} |F^\hbar_{\bm,\bn}|
\\
&\leq
C_{\ell,n} \hbar^{\varepsilon n-d/2} \|F^\hbar\|_{\ell^2(\Z^{2d})},
\end{align*}
where we used the Cauchy-Schwarz inequality and the fact that
$|\Lambda_{\hbar,{\rm rhs}}| \leq C \hbar^{-d}$ due to \eqref{eq_diam_LBA}.
As a result, using the lower-bound in \eqref{tmp_bound_pmn_LBA_mid}, we have
\begin{equation*}
\frac{1}{|\pp_\hbar(\bm,\bn)|^\ell}
|(f_\hbar,(P_\hbar^\star-\overline{\pp}_\hbar(\bm,\bn))^{\ell-1} \gshmn)|
\leq
C_{\ell,n} \hbar^{\varepsilon n-d/2-\ell/2} \|F^\hbar\|_{\ell^2(\Z^{2d})},
\end{equation*}
and
\begin{equation*}
\sum_{\ell=1}^r
\frac{1}{|\pp_\hbar(\bm,\bn)|^\ell}
|(f_\hbar,(P_\hbar^\star-\overline{\pp}_\hbar(\bm,\bn))^{\ell-1} \gshmn)|
\leq
C_{m,n} \hbar^{\varepsilon n-d/2-r/2} \|F^\hbar\|_{\ell^2(\Z^{2d})},
\end{equation*}
for all $n \in \N$. Thus, for any $m \in \N$, we can select $n = n(m,d,r,\varepsilon)$ such that
$\varepsilon n - (d+r)/2 \geq m$, leading to
\begin{equation*}
\sum_{\ell=1}^r
\frac{1}{|\pp_\hbar(\bm,\bn)|^\ell}
|(f_\hbar,(P_\hbar^\star-\overline{\pp}_\hbar(\bm,\bn))^{\ell-1} \gshmn)|
\leq
C_{\varepsilon,m,r} \hbar^{m} \|F^\hbar\|_{\ell^2(\Z^{2d})}.
\end{equation*}

On the other hand, using again \eqref{eq_diam_LBA} and applying
\eqref{eq_iterated_adjoint_operator}, we have
\begin{equation*}
\|(P_\hbar^\star-\overline{\pp_\hbar}(\bm,\bn))^r \gshmn\|_{L^2(\R^d)}
\leq
C_r \hbar^{r/2},
\end{equation*}
and the lower bound in \eqref{tmp_bound_pmn_LBA_mid} shows that
\begin{equation*}
\frac{1}{|\pp_\hbar(\bm,\bn)|^r}\|
(P_\hbar^\star-\overline{\pp_\hbar}(\bm,\bn))^r \gshmn\|_{L^2(\R^d)}
\leq
C_r \hbar^{\varepsilon r}.
\end{equation*}
We then write that
\begin{multline*}
\frac{1}{|\pp_\hbar(\bm,\bn)|^r}
|(u_\hbar,(P_\hbar^\star-\overline{\pp_\hbar}(\bm,\bn))^r \gshmn)|
\leq
C_r \hbar^{\varepsilon r} \|u_\hbar\|_{L^2(\R^d)}
\leq
\\
C_r \hbar^{\varepsilon r-N} \|f_\hbar\|_{L^2(\R^d)}
\leq
C_r \hbar^{\varepsilon r-N} \|F^\hbar\|_{\ell^2(\Z^{2d})}
\leq
C_{\varepsilon,m} \hbar^m \|F^\hbar\|_{\ell^2(\Z^{2d})},
\end{multline*}
up to picking $r$ such that $\varepsilon r - N \geq m$. \qed
\end{proof}

\begin{proof}[Proof of Theorem \ref{theorem_approximability}]
We expand $u_\hbar$ in the frame $(\gshmn)_{[\bm,\bn] \in \Z^{2d}}$ as
\begin{equation*}
u_\hbar = u_\hbar^{\rm near} + u_\hbar^{\rm mid} + u_\hbar^{\rm far}
\end{equation*}
where
\begin{align*}
u_\hbar^{\rm near}
&\eq
\sum_{[\bm,\bn] \in \Lambda_{\hbar,{\rm near}}}
(u_{\hbar},\gshmn^\star)\gshmn
\\
u_\hbar^{\rm mid}
&\eq
\sum_{[\bm,\bn] \in \Lambda_{\hbar,{\rm mid}}}
(u_{\hbar},\gshmn^\star)\gshmn
\\
u_\hbar^{\rm far}
&\eq
\sum_{[\bm,\bn] \in \Lambda_{\hbar,{\rm far}}}
(u_{\hbar},\gshmn^\star)\gshmn.
\end{align*}
The proof then consists in showing that $u_\hbar^{\rm mid}$ and $u_\hbar^{\rm far}$ are small.

{\bf Step 1.} We first treat the $u_\hbar^{\rm mid}$ term. To do so,
we start by introducing the approximation
\begin{equation}
\label{tmp_def_tilde_umid}
\widetilde u_\hbar^{\rm mid}
\eq
\sum_{[\bm,\bn] \in \Lambda_{\hbar,{\rm mid}}}
\sum_{\substack{[\bm'\bn'] \in \Z^{2d} \\ |[\bm,\bn]-[\bm',\bn']| \leq \hbar^{-\varepsilon/2}}}
U^{\bm,\bn}_{\bm',\bn'}
(u_{\hbar},\gshmnp) \gshmn.
\end{equation}
Recalling \eqref{eq_LBA_mid_star_inclusion}, all the $[\bm',\bn']$ indices in the sum belong to the
enlarged set $\Lambda_{\hbar,{\rm mid}}^\star$, so that
\begin{equation}
\label{tmp_u_test_mid_star}
|(u_{\hbar},\gshmnp)| \leq C_{\varepsilon,m} \hbar^m \|F^{\hbar}\|_{\ell^2(\Z^{2d})}
\end{equation}
by Lemma \ref{lemma_LBA_mid_star}.
Recalling \eqref{eq_diam_LBA}, $|[\bm,\bn]| \leq C\hbar^{-1/2}$ for
all $[\bm,\bn] \in \Lambda_{\hbar,{\rm mid}}$, and we have from \eqref{eq_norm_gskmn}
\begin{equation}
\label{tmp_gauss_hp_norm}
\|\gshmn\|_{\HH_\hbar^p(\R^d)} \leq C_p.
\end{equation}
Thus, plugging \eqref{tmp_u_test_mid_star} and \eqref{tmp_gauss_hp_norm} into
\eqref{tmp_def_tilde_umid}, we have
\begin{align}
\nonumber
\|\widetilde u_\hbar^{\rm mid}\|_{L^2(\R^d)}
&\leq
C_{p,\varepsilon,m} \hbar^{m} \|F^\hbar\|_{\ell^2(\Z^{2d})}
\sum_{[\bm,\bn] \in \Lambda_{\hbar,{\rm mid}}}
\sum_{\substack{[\bm'\bn'] \in \Z^{2d} \\ |[\bm,\bn]-[\bm',\bn']| \leq \hbar^{-\varepsilon/2}}}
1
\\
\label{tmp_norm_tilde_umin}
&\leq
C_{p,\varepsilon,m} \hbar^{m-\varepsilon d^2/2} \|F^\hbar\|_{\ell^2(\Z^{2d})}.
\end{align}

We now estimate the difference between $u_\hbar^{\rm mid}$ and $\widetilde u_\hbar^{\rm mid}$
\begin{align*}
u_\hbar^{\rm mid} - \widetilde u_\hbar^{\rm mid}
=
\sum_{[\bm,\bn] \in \Lambda_{\hbar,{\rm mid}}}
\left (u_\hbar,
\gshmn^\star-
\sum_{\substack{[\bm',\bn'] \in \Z^{2d} \\ |[\bm,\bn]-[\bm',\bn']| \leq \hbar^{-\varepsilon/2}}}
U^{\bm,\bn}_{\bm',\bn'} \gshmnp
\right )
\gshmn,
\end{align*}
so that
\begin{align*}
\|u_\hbar^{\rm mid} - \widetilde u_\hbar^{\rm mid}\|_{\HH_\hbar^p(\R^d)}
\leq
C_p
\|u_\hbar\|_{L^2(\R^d)}
\sum_{[\bm,\bn] \in \Lambda_{\hbar,{\rm mid}}}
\left \|
\gshmn^\star
-
\sum_{\substack{[\bm',\bn'] \in \Z^{2d} \\ |[\bm,\bn]-[\bm',\bn']| \leq \hbar^{-\varepsilon/2}}}
\gshmnp
\right \|_{L^2(\R^d)}
\end{align*}
and it follows from Proposition \ref{proposition_tight_expansion} that
\begin{equation}
\label{tmp_diff_tilde_norm}
\|u_\hbar^{\rm mid} - \widetilde u_\hbar^{\rm mid}\|_{\HH_\hbar^p(\R^d)}
\leq
C_{p,\varepsilon,m} \hbar^{m-d} \|u_\hbar\|_{\HH_\hbar^p(\R^d)}
\leq
C_{p,\varepsilon,m} \hbar^{m-d-N} \|F^\hbar\|_{\ell^2(\Z^{2d})}.
\end{equation}
Then, it follows from \eqref{tmp_norm_tilde_umin} and \eqref{tmp_diff_tilde_norm} that
\begin{equation*}
\|u_\hbar^{\rm mid}\|_{\HH_\hbar^p(\R^d)}
\leq
C_{p,\varepsilon,m} \hbar^m \|F^{\hbar}\|_{\ell^2(\Z^{2d})}
\qquad
\forall m \in \N,
\end{equation*}
up to redefining $m$.   

{\bf Step 2.} We then turn our attention to $u_\hbar^{\rm far}$. On the one hand, we can apply
Proposition \ref{proposition_quasi_orthogonality} with
$\CP_{\hbar,L,\TA} \eq \hbar^{[\bbeta]}\partial^{\bbeta} \circ P_\hbar$,
$\Lambda = \{[\bm,\bn]\} \subset\Lambda_{\hbar,{\rm near}}^\star$ and
$\Lambda' = \Lambda_{h, \mathrm{{far}}}$. Using Lemma \ref{Lem:Index}
and the Cauchy-Schwarz inequality, this gives
\begin{align*}
|(\bx^{\balpha} \partial^{\bbeta} (P_\hbar u_\hbar^{\rm far}),\gshmn)| &= \left| \sum_{[\bm',\bn'] \in \Lambda_{\hbar,{\rm far}}}
(u_{\hbar},\gshmnp^\star)
(\bx^{\balpha}\CP_{\hbar,L,\TA} \gshmnp,\gshmn) \right|\\
&\leq  C_{\balpha, \bbeta,r} \|u_h\|_{L^2(\R^d)} \hbar^r
\quad
\forall r \in \N.
\end{align*}

As a result, since $|\Lambda_{\hbar,{\rm near}}^\star| \leq C\hbar^{-d}$ due to
\eqref{eq_diam_LBA}, we have
\begin{align*}
\sum_{[\bm,\bn] \in \Lambda_{\hbar,{\rm near}}^\star}
|(\bx^{\balpha} \partial^{\bbeta} (P_\hbar u_\hbar^{\rm far}),\gshmn)|^2
&\leq
C_{\balpha,\bbeta,r} \hbar^{2r-d}
\|u_\hbar\|_{L^2(\R^d)}^2
\end{align*}
which we rewrite as
\begin{equation*}
\sum_{[\bm,\bn] \in \Lambda_{\hbar,{\rm near}}^\star}
|(\bx^{\balpha} \partial^{\bbeta} (P_\hbar u_\hbar^{\rm far}),\gshmn)|^2
\leq
C_{\balpha,\bbeta,m} \hbar^{2m} \|u_\hbar\|_{L^2(\R^d)}^2 \qquad \forall m \in \N
\end{equation*}
after changing variables.

On the other hand, if $[\bm,\bn] \in \Lambda_{\hbar,{\rm far}}^\star$, we write that
\begin{equation*}
P_\hbar u_\hbar^{\rm far}
=
f_\hbar - P_\hbar u_\hbar^{\rm near} - P_\hbar u_\hbar^{\rm mid},
\end{equation*}
so that
\begin{equation*}
|(\bx^{\balpha}\partial^{\bbeta}(P_\hbar u_\hbar^{\rm far}),\gshmn)|^2
\leq
C \left (
|(\bx^{\balpha}\partial^{\bbeta} f_\hbar,\gshmn)|^2
+
|(\bx^{\balpha}\partial^{\bbeta}(P_\hbar u_\hbar^{\rm near}),\gshmn)|^2
+
|(\bx^{\balpha}\partial^{\bbeta}(P_\hbar u_\hbar^{\rm mid}),\gshmn)|^2
\right ).
\end{equation*}
We then have
\begin{equation*}
\sum_{[\bm,\bn] \in \Lambda_{\hbar,{\rm far}}^\star}
|(\bx^{\balpha}\partial^{\bbeta}(P_\hbar u_\hbar^{\rm mid}),\gshmn)|^2
\leq
C \|\bx^{\balpha}\partial^{\bbeta}(P_\hbar u_\hbar^{\rm mid})\|_{L^2(\R^d)}^2
\leq
C \|u_\hbar^{\rm mid}\|_{\HH_\hbar^{[\bbeta]+2}(\R^d)}^2
\leq
C_{\bbeta,m} \hbar^m \|F^\hbar\|_{\ell^2(\Z^{2d})}^2
\end{equation*}
due to {\bf Step 1}. Next, thanks to Proposition \ref{proposition_quasi_orthogonality},
\begin{align*}
\sum_{[\bm,\bn] \in \Lambda_{\hbar,{\rm far}}^\star}
|(\bx^{\balpha}\partial^{\bbeta} f_\hbar,\gshmn)|^2
&\leq
\sum_{[\bm,\bn] \in \Lambda_{\hbar,{\rm far}}^\star}
\sum_{[\bm',\bn'] \in \Lambda_{\hbar,{\rm rhs}}}
|F^\hbar_{\bm,\bn}|^2|(\bx^{\balpha} \partial^{\bbeta} \gshmnp,\gshmn)|^2
\\
&\leq
C_{\balpha,\bbeta,m}^2 \hbar^{2m} \|F^\hbar\|_{\ell^2(\Z^{2d})}^2.
\end{align*}
Finally, again by Proposition \ref{proposition_quasi_orthogonality}, we have
\begin{align*}
\sum_{[\bm,\bn] \in \Lambda_{\hbar,{\rm far}}^\star}
|(\bx^{\balpha}\partial^{\bbeta} (P_\hbar u_\hbar^{\rm near}),\gshmn)|^2
&\leq
\sum_{[\bm,\bn] \in \Lambda_{\hbar,{\rm far}}^\star}
\sum_{[\bm',\bn'] \in \Lambda_{\hbar,{\rm near}}}
|(u,\gskmnp^\star)||(\bx^{\balpha}\partial^{\bbeta} (P_\hbar \gshmnp),\gshmn)|^2
\\
&\leq
C_{\balpha,\bbeta,r} \hbar^{2r} \|u_\hbar\|_{L^2(\R^d)}^2
\leq
C_{\balpha,\bbeta,m} \hbar^{2m} \|F^\hbar\|_{\ell^2(\Z^{2d})}^2.
\end{align*}

Since $\Lambda_{\hbar,{\rm near}}^\star$ and $\Lambda_{\hbar,{\rm far}}^\star$
form a partition of $\Z^{2d}$, we have thus shown that
\begin{equation}
\label{tmp_regularity_ffar}
\|\bx^{\balpha}\partial^{\bbeta}(P_\hbar u_\hbar^{\rm far})\|_{L^2(\R^d)}
\leq
C_{\balpha,\bbeta,m} \hbar^m \|F^{\hbar}\|_{\ell^2(\Z^{2d})}.
\end{equation}
Letting $f_\hbar^{\rm far} = P_\hbar u_\hbar^{\rm far}$, we see from
\eqref{tmp_regularity_ffar} that $u_\hbar^{\rm far}$ solves
$P_\hbar u_\hbar^{\rm far} = f_\hbar^{\rm far}$ with a right-hand side
$f_\hbar^{\rm far} \in \HH_\hbar^p(\R^d)$ such that
\begin{equation*}
\|f_\hbar^{\rm far}\|_{\HH_\hbar^p(\R^d)}
\leq
C_{p,m} \hbar^m \|F^{\hbar}\|_{\ell^2(\Z^{2d})}
\quad
\forall m \in \N.
\end{equation*}
Then, we conclude the proof with \eqref{eq_polynomial_resolvant}. \qed
\end{proof}

\section{Application to the Helmholtz equation}
\label{sec:helmholtz}

We now turn our attention to the model problem of the Helmholtz equation
\eqref{eq_helmholtz_intro}, with a particular focus on the case of plane-wave scattering.

\subsection{Notation}

In this section $k$ will denote the wavenumber in the Helmholtz problem,
For the sake of simplicity, we assume that $kR \geq 1$. We will apply the
results of Section \ref{sec:setting} with $\hbar \sim (kR)^{-1}$. As a result,
the norms
\begin{equation}
\label{eq_weighted_norm_k}
\|v\|_{\HH_k^p(\R^d)}^2
\eq
\sum_{\substack{\balpha \in \N^d \\ [\balpha] \leq p}}
\sum_{q \leq p - [\balpha]}
k^{-2[\balpha]} \left\|\left|\frac{\bx}{R}\right|^q \partial^{\balpha} v\right\|_{L^2(\R^d)}^2
\end{equation}
will be convenient. Notice that if
\begin{equation*}
\CF(v)(\bxi)
\eq
(2\pi)^{-d/2} \int_{\R^d} v(\bx) e^{-i \bx \cdot \bxi} \mathrm{d}\bx,\, \mbox{ a.e. }\bxi \in \R^d.
\end{equation*}
is the Fourier transform of $v \in L^2(\R^d)$, and if we define the ``reverse'' norm by 
\begin{equation}
\|v\|_{\check{H}_k^p(\R^d)}^2
\eq
\sum_{\substack{\balpha \in \N^d \\ [\balpha] \leq p}}
\sum_{q \leq p - [\balpha]}
k^{-2q} \left\|\left|\frac{\bx}{R}\right|^q \partial^{\balpha} v\right\|_{L^2(\R^d)}^2,
\end{equation}
then we have
\begin{equation}
\label{eq:TFNormePoids}
c(R)
\|\CF v\|_{\check{H}^p_k(\R^d)}
\leq
\|v\|_{\HH^p_k(\R^d)}
\leq
C(R) \|\CF v\|_{\check{H}^p_k(\R^d)}.
\end{equation}
We will also use the following ``standard'' norm
\begin{equation*}
\|v\|_{H_k^p(\R^d)}^2
\eq
\sum_{\substack{\balpha \in \N^d \\ [\balpha] \leq p}}
k^{-2[\balpha]}\|\partial^{\balpha} v\|_{L^2(\R^d)}^2.
\end{equation*}

\subsection{Model problem}
\label{section_model_problem}

\begin{subequations}
\label{eq_helmholtz}
We consider smooth coefficients $\mu,\BA \in C_{\rm b}^\infty(\R^d)$
that are respectively equal to $1$ and $\BI$ outside $B(0,R)$. Given
$f: \Omega \to \mathbb C$ our model problem is to find $u: \Omega \to \mathbb C$
such that
\begin{equation}
\label{eq_helmholtz_volume}
-k^2 \mu u - \div \left (\BA\grad u\right ) = f \text{ in } \R^d
\end{equation}
and
\begin{equation}
\label{eq_sommerfeld}
\frac{\partial u}{\partial |\bx|}(\bx) - ik u(\bx)
=
o\left( |\bx|^{(-d-1)/2}\right)
\text{ as }
|\bx| \to +\infty.
\end{equation}
\end{subequations}
Problem \eqref{eq_helmholtz} is well-posed in the sense that
for all $f \in L^2_{\rm comp}(\R^d)$, there exists a unique
$R_k f \eq u \in L^2_{\rm loc}(\R^d)$ such that \eqref{eq_helmholtz_volume}
and \eqref{eq_sommerfeld} hold true.

Following the assumption on $\hbar$ in Section \ref{section_h_notations}
and assumption \eqref{eq_polynomial_resolvant}, we focus on wave numbers such that
$k \in \LK \subset [1/R,+\infty)$, and assume that there exists $N \geq 1$ such that
\begin{equation}
\label{eq_resolvant_helmholtz}
k^2 \|\cutoff R_k \cutoff\|_{L^2(\R^d) \to L^2(\R^d)}
\leq
C (kR)^N
\qquad
\forall k \in \LK
\end{equation}
where $\cutoff$ is a smooth cut-off function that takes the value $1$ on $B(0,R)$
and $0$ outside $B(0,2R)$.

Notice that in view of the term $-k^2\mu u$, in \eqref{eq_helmholtz_volume},
the factor $k^2$ in \eqref{eq_resolvant_helmholtz} ensures that the right-hand side
does not bear any physical dimension. We have explicitly required that $N \geq 1$ for
the sake of clarity, but in fact, it can be shown that $N = 1$ is the best possible
exponent for which \eqref{eq_resolvant_helmholtz} can hold true, see e.g.
\cite[Lemma 4.1]{chaumontfrelet_gallistl_nicaise_tomezyk_2018a}.

\begin{remark}[When does the polynomial bound actually hold?]
\label{rem:ResolEst}
The bound \eqref{eq_resolvant_helmholtz} is known to hold in several situations:
\begin{itemize}
\item
When the dynamics induced by the Hamiltonian $p$ has no trapped trajectory, i.e.,
when every trajectory leaves any compact set in finite time, the assumption holds
with $\LK = [1/R, +\infty)$. See for instance \cite{galkowski2019optimal}.
This situation is often referred to as ``non-trapping''.
\item
When the dynamics induced by $p$ has a trapped set, and the dynamics is hyperbolic,
close to this trapped set, it has been conjectured in \cite{zworski2017mathematical}
that \eqref{eq_polynomial_resolvant} always holds with $\LK = [1/R, +\infty)$.
Actually, this is already known when the trapped set is ``filamentary enough'', see
\cite{nonnenmacher2011spectral,nonnenmacher2009quantum}.
\item
Without any assumption on the dynamics, \eqref{eq_polynomial_resolvant} holds taking
$\LK$ to be $[1/R, +\infty)$ from which we exclude a set of frequencies $k$ whose
intersection with $\{k \geq k_0\}$ has a length going to zero as $k_0 \to +\infty$. We
refer the reader to \cite{lafontaine_spence_wunsch_2019a} for more details.
\end{itemize}
\end{remark}

\subsection{Perfectly matched layers}\label{sec:PML}

As advertised in the introduction, the formulation \eqref{eq_helmholtz} is not
suited for immediate discretization by ``volume'' methods, as the radiation condition
is hard to take into account. We will thus rely on an equivalent formulation that uses
perfectly matched layers (PML).

Specifically, given $f: \Omega \to \C$, we consider the problem to find
$u: \Omega \to \C$ such that $P_k u = f$ where
\begin{equation}
\label{eq:PML}
P_k u
\eq
-\frac{1}{k^2}
\left((\BI+ i \BM)^{-1} \grad \right) \cdot \left((\BI+ i \BM)^{-1} \BA \grad u \right)
-\mu u.
\end{equation}
In \eqref{eq:PML}, the (SPD) matrix function $\BM$ is given by
\begin{equation*}
\BM(\bx)
\eq
\frac{g(|\bx|)}{|\bx|^3}
\left (
|\bx|^2 \mathrm{Id} - \bx\otimes \bx
\right )
+
\frac{g'(|\bx|)}{|\bx|^2} \bx \otimes \bx,
\end{equation*}
where $g : \R \longrightarrow \R$ is a user-defined function such that $g(r) = 0$ if
$r \leq R$, $g(r) = r$ if $r\geq R_0 > R$, and $g'(r) \geq 0$.
In what follows, we will assume that $g$ is a smooth function to
satisfy the assumptions of Section \ref{sec:setting}, but many
results about PML still hold with less regular $g$ (see, e.g., \cite{galkowski2021perfectly}).

Notice that, if $|\bx| \leq R$, $\BM = \boldsymbol 0$, so that the original
operator is not modified on the support of $\mu-1$ and $\BA-\BI$. On the other
hand, $\BM = \BI$ if $|\bx| \geq R_0$, so that dissipation is introduced away
from the origin, where the operator simply reads
\begin{equation*}
P_kv = \frac{i}{2k^2} \Delta v - v
\qquad \text{ whenever } \qquad
\operatorname{supp} v \cap B(0,2R) = \emptyset.
\end{equation*}
This transformation can be naturally interpreted as a complex deformation
of coordinates. It is also often called the complex scaling technique.
Crucially, the PML is designed in such way that
\begin{equation}
\label{eq_pml_exact}
\left ((P_k)^{-1} f\right )|_{B(0,R)} = \left (R_k f\right )|_{B(0,R)}
\end{equation}
whenever $\operatorname{supp} f \subset B(0,R)$.
We refer the reader to \cite[\S 4.5]{dyatlov2019mathematical} or \cite{galkowski2021perfectly}
for more information.

\subsection{Abstract setting}

We now verify that the Helmholtz problem formulated with PML indeed fits
the abstract setting of Section \ref{ssec:general}. The only non-trivial
facts to establish are the polynomial resolvent estimates in $\HH^q(\R^d)$
in \eqref{eq_polynomial_resolvant} and the boundedness of the energy layer
\eqref{eq_assumption_symbol_bounded}.

\begin{lemma}[Resolvent estimates]
Let $q \in \N$. For all $f \in \HH^q(\R^d)$, there exists a unique $u \in L^2(\R^d)$
such that $P_k u = f$. In addition, we have $u \in \HH^q(\R^d)$ with
\begin{equation*}
k^2\|u\|_{\HH^q_k(\R^d)}
\leq
C (kR)^{N} \|f\|_{\HH^q_k(\R^d)}.
\end{equation*}
Furthermore, if $\operatorname{supp} f \subset B(0,R)$, then $u \in \HH^{q+2}(\R^d)$ with
\begin{equation}\label{eq:ResolvantePoly}
k^2\|u\|_{\HH^{q+2}_k(\R^d)}
\leq
C (kR)^{N} \|f\|_{\HH^q_k(\R^d)}.
\end{equation}
\end{lemma}

\begin{proof}
We first invoke Theorem 1.6 of \cite{galkowski2021perfectly} which states that
\begin{equation*}
\|P_k^{-1}\|_{L^2(\R^d) \to H^2_k(\R^d)}
\leq
C \|\cutoff R_k \cutoff\|_{L^2(\R^d) \to L^2(\R^d)}.
\end{equation*}
Then, using a usual bootstrap technique, we easily show that
\begin{equation}\label{eq:Bootstrap}
\|P_k^{-1}\|_{H^q_k(\R^d) \to H^{q+2}_k(\R^d)}
\leq
C \|\cutoff R_k \cutoff\|_{L^2(\R^d) \to L^2(\R^d)}
\end{equation}
for all $q \in \N$.

We then need to take care of the weights in the $\HH^q(\R^d)$ norms.
To do so, we observe that if $P_k u = f$,  then we may write
\begin{equation*}
2ik^2 u + \Delta u = g,
\end{equation*}
with
\begin{equation*}
g \eq -2 i k^2 P_k u + (2i k^2P_k u + \Delta u + 2ik^2) u = - 2ik^2(f + Q_k u),
\end{equation*}
where $Q_k$ is differential operator of degree 2 with smooth coefficients
supported in $B(0,R_0)$. Let $\balpha,\bbeta \in \N^d$ with $|\balpha|,|\bbeta| \leq 2$.
Since $Q_k$ is compactly supported, we have
\begin{equation}\label{eq:Controleg}
\left\|\left(\frac{\bx}{R}\right)^{\balpha} g\right\|_{\HH^q_k(\R^d)}
\leq
C_{\balpha} k^2 \left(\|u\|_{H_k^{q+2}(\R^d)}+ \|f\|_{\HH^{q+2}_k(\R^d)}\right).
\end{equation}

Now, we note that
\begin{equation}\label{eq:TFu}
\left( \frac{\bx^{\balpha}}{R}\right)  \partial^{\bbeta}   u
= i^{\alpha - \beta}
\CF^{-1} \left( \left(\frac{1}{R}\partial\right)^{\balpha} \left( \frac{\bxi^{\bbeta} }{-|\bxi|^2 + 2ik^2} \CF(g) \right)\right),
\end{equation}
Let us write
\begin{equation*}
\frac{\bxi^{\bbeta}}{-|\bxi|^2 + 2ik^2}
=
k^{[\bbeta]-2} \frac{(\bxi/k)^{\bbeta}}{-|\bxi/k|^2 + 2i},
\end{equation*}
so that the map
\begin{equation*}
\bxi \mapsto \frac{\bxi^{\bbeta}}{-|\bxi|^2 + 2ik^2}
\end{equation*}
has $C^\ell$ norm bounded by $C_\ell k^{[\bbeta] - 2- \ell} \leq C_\ell k^{[\bbeta] - 2}$,
since $k \geq 1$.

We deduce from \eqref{eq:TFNormePoids} and \eqref{eq:Controleg} that 
\begin{equation*}
\left \| \left(\frac{1}{R}\partial\right)^{\balpha} \left(
\frac{\bxi^{\bbeta}}{-|\bxi|^2 + 2ik^2} \CF( g) \right)
\right \|_{\check{H}^q_k(\R^d)}
\leq
Ck^{[\bbeta]}
\left(\|u\|_{H_k^{q+2}(\R^d)}+ \|f\|_{\HH^{q+2}_k(\R^d)}\right)
\end{equation*}
and hence, thanks to \eqref{eq:TFu},
\begin{multline*}
\|u\|_{\HH^{q+2}_k(\R^d)}
\leq
C
\sum_{[\balpha] + [\bbeta] \leq 2} k^{-[\bbeta]}
\|\left(\frac{\bx}{R}\right)^{\balpha} \partial^{\bbeta} u\|_{\HH_k^p(\R^d)}
\leq
C \sum_{[\balpha] + [\bbeta] \leq 2} k^{-[\bbeta]}
\left\|\mathcal{F} \left (
\left(\frac{\bx}{R}\right)^{\balpha} \partial^{\bbeta} u
\right )\right\|_{\check{H}_k^p(\R^d)}
\\
\leq
C \sum_{[\balpha] + [\bbeta] \leq 2}
\left(\|u\|_{H_k^{q+2}(\R^d)}+ \|f\|_{\HH^{q+2}_k(\R^d)}\right)
\leq
C \|\cutoff R_k \cutoff\|_{L^2(\R^d) \to L^2(\R^d)}  \|f\|_{\HH^{q+2}_k(\R^d)},
\end{multline*}
thanks to \eqref{eq:Bootstrap}.
The first part of the result follows from \eqref{eq_resolvant_helmholtz}.

Now, if we further assume that $f$ is supported in $B(0,R)$,
$g$ is compactly supported and equation \eqref{eq:Controleg}
may be replaced by
\begin{equation*}
\left\|\left(\frac{\bx}{R}\right)^{\balpha} g\right\|_{\HH^q_k(\R^d)}
\leq
C_{\balpha} k^2  \left (\|f\|_{H^q_k(\R^d)} + \|u\|_{H_k^{q+2}(\R^d)} \right ),
\end{equation*}
and the same reasoning as above leads to
\begin{equation*}
\|u\|_{\HH^{q+2}_k(\R^d)}
\leq C  \left(\|u\|_{H_k^{q+2}(\R^d)}+ \|f\|_{H^{q}_k(\R^d)}\right)
\leq C \|\cutoff R_k \cutoff\|_{L^2(\R^d) \to L^2(\R^d)}  \|f\|_{\HH^{q}_k(\R^d)},
\end{equation*}
as announced. \qed
\end{proof}

\begin{lemma}[Boundedness of the energy layer]
For all $\delta \in (0,1)$, the set
\begin{equation*}
\left \{
[\bx,\bxi] \in \R^{2d} \; | \; |p(\bx,\bxi)| \leq \delta
\right \}
\end{equation*}
is bounded.
\end{lemma}

\begin{proof}
Fix $0 < \delta < 1$, and let $U \eq \{ [\bx,\bxi] \in \R^{2d}; \; |p(\bx,\bxi)| \leq \delta \}$.
Let $\bx \in \R^d$. We first assume that $|\bx| \leq R_0$. Since we know that
\begin{equation*}
C|\bxi|^2-C' \leq |p(\bx,\bxi)| \leq \delta,
\end{equation*}
we see that $\{\bxi\in \R^d \; |p(\bx,\bxi)| \leq \delta \}$
is bounded,  and thus $U \cap (B_{\bx}(0,R_0)\times \R^d)$ is bounded.
On the other hand, if $|\bx| \geq R_0$, we have
\begin{equation*}
p(\bx,\bxi)
=
(1+i)^2 |\bxi|^2-1
=
2i |\bxi|^2-1
\end{equation*}
so that
\begin{equation*}
|p(\bx,\bxi)|^2 = 4|\bxi|^4 + 1 \geq 1 > \delta,
\end{equation*}
which implies that $U \setminus (B_{\bx}(0,R_0)\times \R^d) = \emptyset$. \qed
\end{proof}

We finally show that the right-hand side associated with plane-wave scattering
are indeed well-approximated by Gaussian coherent states in order to apply
Corollary \ref{corollary_approximability} later on.

\begin{lemma}[Approximability of plane-wave right-hand sides]
\label{Lem:ApproxPW}
For $k > 0$, consider the right-hand side
$f_k \eq (-k^2-\Delta)(\chi e^{ik\bd\cdot \bx})$ where
$\chi \in C^\infty_{\rm c}(B_R)$ and $\bd \in \R^d$
with $|\bd| = 1$. Then, there exists $F^k \in \ell^2(\Z^{2d})$
such that
\begin{equation}
\label{eq_estimate_planewave}
R^2
\left \|
f_k - \sum_{[\bm,\bn] \in \Lambda_{k,{\rm rhs}}} F^k_{\bm,\bn} \gskmn
\right \|_{\HH^p_k(\R^d)}
\leq
C_{\varepsilon,p,m} (kR)^{-m},
\end{equation}
for all $m \in \N$, where
\begin{equation*}
\Lambda_{k,{\rm rhs}}^{\prime}
\eq
\left \{
[\bm, \bn]\in \Z^{2d}
\; | \;
\mathrm{dist}\left(\bx^{k,\bm}, \mathrm{supp}(\chi)\right) \leq (kR)^{-1/2+\varepsilon}
\text{ and }
\left||\bxi^{k,\bn}|-1\right| \leq  (kR)^{-1/2+\varepsilon}
\right \}.
\end{equation*}
In particular, if $p$ is the symbol of the function appearing in section \ref{sec:PML} and if
$\chi$ is supported in the region where $p(\bx,\bxi) = |\bxi|^2-1$, then
$\Lambda_{k,{\rm rhs}}^{\prime}$ is of the the same form as $\Lambda_{k,{\rm rhs}}$
introduced in \eqref{eq:FormeRHS} with $\alpha=0$.
\end{lemma}

\begin{proof}
We start by observing that
\begin{equation*}
f_k
=
-(\Delta \chi + 2ik\bd\cdot\grad \chi)e^{ik\bd\cdot\bx}
=
-R^{-2}(R^2\Delta \chi + 2ikR(R \bd\cdot\grad \chi))
e^{ik\bd\cdot\bx},
\end{equation*}
so that the result holds true if we can show \eqref{eq_estimate_planewave} for
$\tf_k \eq \tchi e^{ik\bd\cdot\bx}$ with $\tchi$ smooth and compactly supported
(without the factor $R^2$ in front of the left-hand side).

From now on, we fix $\tchi$, $p$, $m$ and $\varepsilon$,
and consider $\tf_k$ as above. We will first show that, if
$[\bm,\bn] \notin \Lambda_{k,{\rm rhs}}^{\prime}$, then we have
\begin{equation}
\label{eq:DecayCoefPW}
\left| (\tf_k, \Psi_{k,\bm,\bn})\right| \leq C_{m}  (kR)^{-m}
\quad \forall m \in \N.
\end{equation}
The quantity $(\tf_k, \Psi_{k,\bm,\bn})$ is of the form
\begin{equation*}
\int_{\R^d} \tchi(\bx) e^{ik \varphi_{\bm,\bn}(\bx)}\mathrm{d}\bx,
\end{equation*}
with 
\begin{equation*}
\varphi_{\bm,\bn}(\bx)
=
\bx \cdot \left(\bxi - \bxi^{k,\bn}\right)
+
\frac{i}{2}|\bx-\bx^{k,\bm}|^2,
\end{equation*}
so that
\begin{equation*}
\nabla \varphi_{\bm,\bn}(\bx) = \bxi - \bxi^{k,\bn}+ i (\bx- \bx^{\bm,k}).
\end{equation*}
In particular, we have
\begin{equation*}
|\nabla \varphi_{\bm,\bn}(\bx)|
\geq
\left||\bxi^{k,\bn}|-1\right|
+
\mathrm{dist}\left (\bx^{k,\bm}, \mathrm{supp}(\chi)\right )
\geq
\frac{1}{2} (kR)^{-\frac{1}{2}+ \varepsilon}.
\end{equation*}
we may use the method of non-stationary phase (i.e., integrate by parts several times,
as in \cite[Lemma 3.14]{zworski2012semiclassical}) to deduce \eqref{eq:DecayCoefPW}.

Now, combining \eqref{eq:DecayCoefPW} with Proposition
\ref{proposition_tight_expansion}, we see that there exists $C$
such that, for any $[\bm,\bn]\in \Lambda_{k,{\rm rhs}}$, we have
\begin{equation}
\label{eq:DecayPWStar}
\left | (\tf_k, \gskmn^\star) \right | \leq C_{m} (kR)^{-m} \quad \forall m \in \N.
\end{equation}
On the other hand, it follows from \cite[Theorem 3.1]{chaumontfrelet_ingremeau_2022a} that we have
\begin{equation}\label{eq:Rappel:TheoetMaxSontTropCools}
\left \|
\tf_k
-
\sum_{\substack{[\bm, \bn]\in \Z^{2d}\\{|\bm,\bn]| \leq k}}} (f_k,\gskmn^\star) \gskmn
\right\|_{\HH_k^p(\R^d)}
\leq
C_{p,m} (kR)^{-m}.
\end{equation}
Writing 
\begin{align*}
\tf_k
&=
\sum_{[\bm, \bn]\in \Lambda_{k,  \rm rhs}} (\tf_k,\gskmn^\star) \gskmn
+
\sum_{\substack{[\bm, \bn]\in \Z^{2d}\setminus \Lambda_{k,  \rm rhs}\\{|\bm,\bn]| \leq k}}}
(\tf_k,\gskmn^\star) \gskmn
\\
&+
\left (
\tf_k
-
\sum_{\substack{[\bm, \bn]\in \Z^{2d}\\{|\bm,\bn]| \leq k}}} (f_k,\gskmn^\star)
\gskmn \right ),
\end{align*}
we deduce from \eqref{eq:DecayPWStar} and \eqref{eq:Rappel:TheoetMaxSontTropCools}
that the last two terms have a $\HH_k^p$ norm bounded by $C_{\varepsilon,p,m} (kR)^{-m}$,
and the result follows. \qed
\end{proof}

\subsection{Approximability estimates}

We are now ready to present our approximability estimates for the Helmholtz problem.
We start with an approximation result for general right-hand sides that does not hinge
on Section \ref{sec:setting}, but rather on standard results on Gabor frames and modulation
spaces, see e.g. \cite[Chapter 11]{Gro}. We also refer the reader to the research note
\cite{chaumontfrelet_ingremeau_2022a} where elementary proofs are presented.

Recall that $N \geq 1$ was introduced in \eqref{eq_resolvant_helmholtz} and that $kR \geq 1$.

\begin{theorem}[Approximability at a fixed frequency]
\label{Th:ApproxKFixed}
Assume $f \in \HH^p(\R^d)$ with $\operatorname{supp} f \subset B(0,R)$. If
\begin{equation*}
\Lambda \eq \left \{
[\bm,\bn] \in \Z^{2d}
\; | \;
|[\bm,\bn]| \leq \sqrt{\rho kR}
\right \},
\end{equation*}
for some $\rho>0$, then we have
\begin{equation*}
k^2\left \|
u-\sum_{[\bm,\bn] \in \Lambda} (u,\gskmn^\star)\gskmn
\right \|_{\HH^q_k(\R^d)}
\leq
C_p
(kR)^{N-(2+p-q)}
\rho^{-(2+p-q)/2} \|f\|_{\HH_k^p(\R^d)}
\end{equation*}
for all $p,q \in \N$ with $q \leq p$.
\end{theorem}

\begin{proof}
Let us set $D \eq (kR/\pi)^{1/2} (\rho kR)^{1/2} = \sqrt{\rho/\pi}(kR)$.
We start with \cite[Theorem 3.1]{chaumontfrelet_ingremeau_2022a}, showing that
\begin{equation*}
\left \|
u-\sum_{[\bm,\bn] \in \Lambda} (u,\gskmn^\star)\gskmn
\right \|_{\HH^q_k(\R^d)}
\leq
C_p D^{q-p-2}
\|u\|_{\HH^{p+2}_k(\R^d)}.
\end{equation*}
The result then follows using \eqref{eq:ResolvantePoly}, since
\begin{multline*}
D^{q-p-2} k^2\|u\|_{\HH^{p+2}_k(\R^d)}
\leq
C_p (\sqrt{\rho} kR)^{q-p-2}
k^2\|u\|_{\HH^{p+2}_k(\R^d)}
\\
\leq
C_p (\sqrt{\rho} kR)^{q-p-2}
(kR)^{N}
\|f\|_{\HH^{p}_k(\R^d)}
\\
\leq
C_p
(kR)^{N+q-p-2}
\rho^{(q-p-2)/2}
\|f\|_{\HH^{p}_k(\R^d)}.
\qed
\end{multline*}
\end{proof}

Our second approximability estimate applies specifically to high-frequency scattering
problems. It is a direct consequence of Theorem \ref{theorem_approximability} and
Lemma \ref{Lem:ApproxPW}.

\begin{theorem}[Approximability in the high-frequency regime]
\label{Theo:ApproxHF}
Fix $0<\varepsilon< 1/2$ and consider the index set
\begin{equation*}
\Lambda
\eq
\left \{
[\bm,\bn] \in \Z^{2d} \; | \;
|p(\bm,\bn)| \leq (kR)^{-1/2+\varepsilon}
\right \}.
\end{equation*}
Then, if the right-hand side is of the form $f_k \eq (-k^2-\Delta)(\chi e^{ik\bd\cdot\bx})$
where $\bd \in \R^d$ with $|\bd| = 1$ and $\chi \in C^\infty_{c}(B_R)$ is supported in the
region where $\mu = 1$ and $\BA = \BI$, for all $q \in \N$, we have
\begin{equation*}
\left \|
u_k-\sum_{[\bm,\bn] \in \Lambda} (u_k,\gskmn^\star)\gskmn
\right \|_{\HH^q_k(\R^d)}
\leq
C_{\varepsilon,q,m} (kR)^{-m}
\quad
\forall m \in \N.
\end{equation*}
\end{theorem}

\subsection{A least-squares method}
\label{section_least_squares}

In this section, we introduce a least-squares method based on Gaussian coherent states.
For a finite set $\Lambda\subset \Z^{2d}$, we consider the discretization space
\begin{equation*}
W_\Lambda
\eq
\mathrm{Vect} \left\{ \gskmn; \; [\bm,\bn] \in \Lambda \right\}.
\end{equation*}
Then, the least squares method consists in finding
$u_\Lambda \in W_\Lambda$ such that
\begin{equation}
\label{eq_least_squares}
\left( P_k u_{\Lambda},P_k w_\Lambda \right) = \left( f, P_k w_{\Lambda} \right)
\end{equation}
for all $w_\Lambda \in W_\Lambda$. Simple manipulations reveal
that \eqref{eq_least_squares} is the Euler-Lagrange equation corresponding
to the minimization problem
\begin{equation}
\|P_k(u-u_\Lambda)\|_{L^2(\R^d)}
=
\min_{w_\Lambda \in W_\Lambda} \|P_k(u-w_\Lambda)\|_{L^2(\R^d)}.
\end{equation}
Since $W_\Lambda$ is finite dimensional, the cost functional is quadratic,
and $P_k$ is an isomorphism, $u_\Lambda$ is uniquely defined.

Using Theorem \ref{Th:ApproxKFixed}  we can easily show that the method
converges for any fixed frequency, if sufficiently many DOFs per wavelength are employed,
as stated in the following corollary:

\begin{corollary}[Convergence at a fixed frequency]
\label{corollary_least_squares_fixed}
Consider the set
\begin{equation*}
\Lambda_{\rho} \eq
\left \{
[\bm,\bn] \in \Z^{2d}
\; | \;
|\bm|^2 + |\bn|^2 \leq \rho kR
\right \},
\end{equation*}
then the following error estimates holds true:
\begin{equation*}
k^2\|u-u_\Lambda\|_{L^2(\R^d)}
\leq
C_p
(kR)^{N-p}
\rho^{-p/2}  \|f\|_{\HH^p(\R^d)}
\quad
\forall p \in \N.
\end{equation*}
\end{corollary}

Relying on Theorem \ref{Theo:ApproxHF} instead, we can show a refined error
estimate in the high-frequency regime.

\begin{corollary}[Convergence in the high-frequency regime]
Fix $0 < \varepsilon < 1/2$,
\begin{equation*}
\Lambda \eq
\left \{
[\bm,\bn] \in \Z^{2d} \; | \;
|p(\xxkm,\xikn)|
\leq
(kR)^{-1/2+\varepsilon}
\right \},
\end{equation*}
and let $f_k \eq (-k^2-\Delta)(\chi e^{ik\bd\cdot\bx)}$ with $\chi \in C^\infty_{\rm c}(B_R)$
supported in the region where $\mu = 1$ and $\BA = \BI$.  Then, if $kR$ is sufficiently large,
we have
\begin{equation*}
k^2\|u-u_\Lambda\|_{L^2(\R^d)}
\leq
C_{\varepsilon,m} (kR)^{-m}
\qquad
\forall m \in \N.
\end{equation*}
\end{corollary}

\section{Numerical results}
\label{sec:numerics}

In this section, we provide numerical illustrations of the above theory
in the one-dimensional case.
The purpose of these examples is simply to illustrate and support our theoretical findings.
Extension to higher dimensions and discussion about efficient linear system assembly and
solve will be reported elsewhere.

\subsection{Setting}

We consider the one dimensional case, where the differential operator in \eqref{eq:PML}
simplifies to
\begin{equation}\label{eq:helm}
P_k v = -\mu\nu u - k^{-2} (\alpha \nu^{-1} v')'
\end{equation}
where $\mu,\alpha > 0$ are smooth ``physical'' coefficients, and $\nu$
is the PML scaling. For the sake of simplicity we take $R \eq 1$, meaning
that $\mu = \alpha = 1$ outside of $(-1,1)$ and that our right-hand sides
$f$ will be supported in $(-1,1)$. Here, we select $\nu \eq 1 + i\sigma$
where $\sigma$ is a stretching function defined as
\begin{equation*}
\sigma(x)
\eq
a\left (\mathbf{1}_{(1,+\infty)} (x-1)^r + \mathbf{1}_{(-\infty,-1)}(-1-x)^r \right )
\end{equation*}
with $a \eq 1/10$ and $r \eq 4$. Notice that this choice slightly departs from our
theoretical framework as the coefficients are only $C^3$ and not $C^\infty$.

In all our experiments, the grid in phase space is chosen as
\begin{equation*}
x^{k,m} = \sqrt{k^{-1}\pi}m, \quad \xi^{k,n} = \sqrt{k^{-1}\pi}n,
\qquad
n,m \in \Z
\end{equation*}
and for a given values of $k$ and $\delta$, the set of indices
is taken to be
\begin{equation*}
\Lambda \eq \left \{
(m,n) \in \Z^2 \; | \; |p(x^{k,m},\xi^{k,n})| < \delta
\right \}.
\end{equation*}

\subsection{Homogeneous medium with analytical solution}

Our first example concerns the case where $\alpha = \mu = 1$.
The right-hand side is given by
\begin{equation*}
\label{eq_numerics_fk}
f_k \eq P_k(\phi e^{ikx}) = (\phi'' + 2ik\phi')e^{ikx},
\end{equation*}
where $\phi$ is the only even function in $C^3(\R^d)$ such that
$\phi = 0$ on $(-\infty,-3/4)$, $\phi = 1$ on $(-1/2,0)$, and
$\phi$ is a polynomial of degree $7$ in $(-3/4,-1/2)$,
and the associated solution is $u_k(x) \eq \phi(x) e^{ikx}$.
Results are reported in Table \ref{table_homogeneous}, where we have chosen
$\delta$ to maintain a constant accuracy for different values of $k$.
Figure \ref{figure_spaces_homogeneous} represents the points $(x^{k,m},\xi^{k,n})$
included in the space $\Lambda$ for different values of $\delta$ for the case $k=400$.

\begin{table}[!h]
\centering
\begin{tabular} {|c|ccc|}
\hline\noalign{\smallskip}
 $k$ & \multicolumn{3}{c|}{Relative $H^1_k$ error ($\delta$) [$N_{\rm dofs}$]} \\
\hline\noalign{\smallskip}
20 &  3.3539e-01 (0.5) [60] & 2.2356e-05 (2.0) [177] & 6.6722e-06 (4.0) [249] \\
50 &  2.3180e-02 (0.4) [38] & 1.7831e-05 (1.0) [229] & 4.4104e-06 (3.0) [419] \\
100 & 4.4322e-01 (0.3) [96] & 3.9496e-05 (0.8) [278] & 2.5610e-06 (2.0) [645] \\
200 & 3.9807e-02 (0.2) [58] & 1.8555e-05 (0.6) [334] & 4.4586e-06 (1.0) [825] \\
400 & 3.6922e-01 (0.1) [74] & 3.6289e-05 (0.4) [532] & 2.3452e-05 (0.6) [646] \\
\hline
\end{tabular}
\caption{Numerical results for different $k$ while varying $\delta$}
\label{table_homogeneous}
\end{table}

\begin{figure}
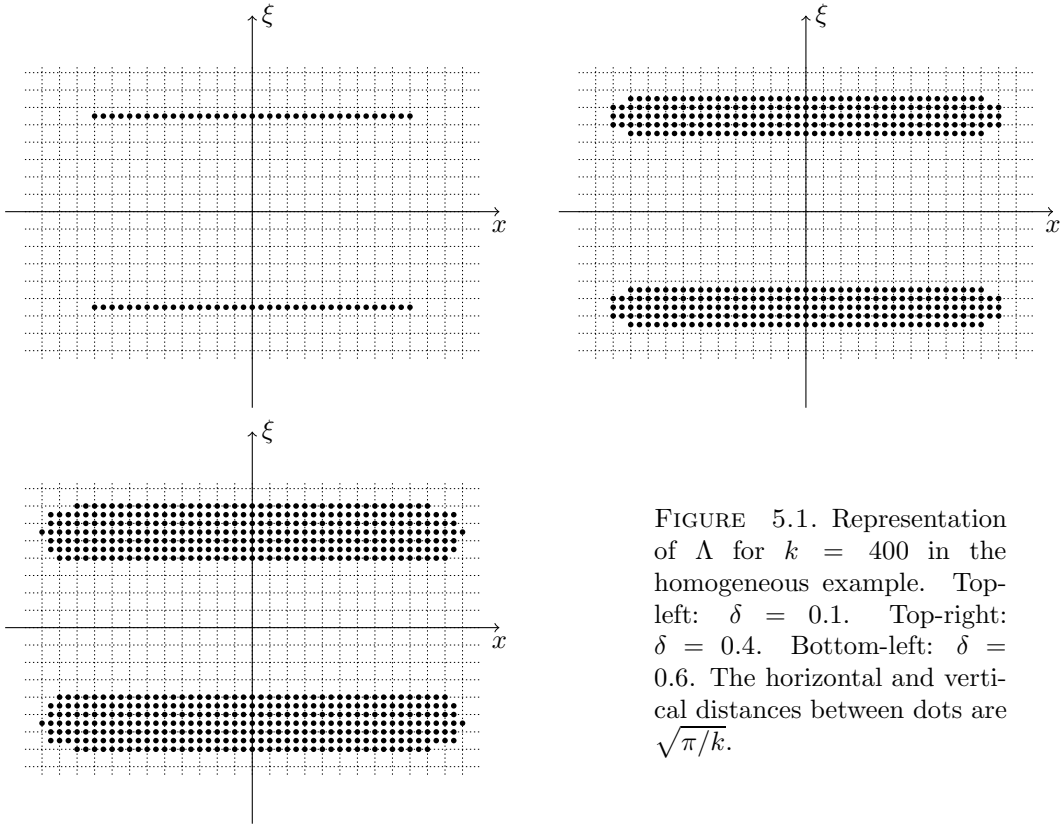

\begin{minipage}{.48\linewidth}
\begin{tikzpicture}[scale=1.3]

\draw[step=.17724538508,dash pattern=on 0.5pt off 1pt,ultra thin] (-2.3,-1.5) grid (2.3,1.5);

\draw[->] (-2.5,0) -- (2.5,0) node[anchor=north] {$x$};
\draw[->] (0,-2.0) -- (0,2.0) node[anchor=west] {$\xi$};

\fill (-1.59520846581496,-0.97484961799) circle (0.8pt);
\fill (-1.59520846581496, 0.97484961799) circle (0.8pt);
\fill (-1.50658577326969,-0.97484961799) circle (0.8pt);
\fill (-1.50658577326969, 0.97484961799) circle (0.8pt);
\fill (-1.41796308072441,-0.97484961799) circle (0.8pt);
\fill (-1.41796308072441, 0.97484961799) circle (0.8pt);
\fill (-1.32934038817914,-0.97484961799) circle (0.8pt);
\fill (-1.32934038817914,0.974849617998) circle (0.8pt);
\fill (-1.24071769563386,-0.97484961799) circle (0.8pt);
\fill (-1.24071769563386,0.974849617998) circle (0.8pt);
\fill (-1.15209500308859,-0.97484961799) circle (0.8pt);
\fill (-1.15209500308859,0.974849617998) circle (0.8pt);
\fill (-1.06347231054331,-0.97484961799) circle (0.8pt);
\fill (-1.06347231054331,0.974849617998) circle (0.8pt);
\fill (-0.97484961799803,-0.97484961799) circle (0.8pt);
\fill (-0.97484961799803,0.974849617998) circle (0.8pt);
\fill (-0.88622692545275,-0.97484961799) circle (0.8pt);
\fill (-0.88622692545275,0.974849617998) circle (0.8pt);
\fill (-0.79760423290748,-0.97484961799) circle (0.8pt);
\fill (-0.79760423290748,0.974849617998) circle (0.8pt);
\fill (-0.70898154036220,-0.97484961799) circle (0.8pt);
\fill (-0.70898154036220,0.974849617998) circle (0.8pt);
\fill (-0.62035884781693,-0.97484961799) circle (0.8pt);
\fill (-0.62035884781693,0.974849617998) circle (0.8pt);
\fill (-0.53173615527165,-0.97484961799) circle (0.8pt);
\fill (-0.53173615527165,0.974849617998) circle (0.8pt);
\fill (-0.44311346272637,-0.97484961799) circle (0.8pt);
\fill (-0.44311346272637,0.974849617998) circle (0.8pt);
\fill (-0.35449077018110,-0.97484961799) circle (0.8pt);
\fill (-0.35449077018110,0.974849617998) circle (0.8pt);
\fill (-0.26586807763582,-0.97484961799) circle (0.8pt);
\fill (-0.26586807763582,0.974849617998) circle (0.8pt);
\fill (-0.17724538509055,-0.97484961799) circle (0.8pt);
\fill (-0.17724538509055,0.974849617998) circle (0.8pt);
\fill (-0.08862269254527,-0.97484961799) circle (0.8pt);
\fill (-0.08862269254527,0.974849617998) circle (0.8pt);
\fill ( 0.00000000000000,-0.97484961799) circle (0.8pt);
\fill ( 0.00000000000000,0.974849617998) circle (0.8pt);
\fill ( 0.08862269254527,-0.97484961799) circle (0.8pt);
\fill ( 0.08862269254527,0.974849617998) circle (0.8pt);
\fill ( 0.17724538509055,-0.97484961799) circle (0.8pt);
\fill ( 0.17724538509055,0.974849617998) circle (0.8pt);
\fill ( 0.26586807763582,-0.97484961799) circle (0.8pt);
\fill ( 0.26586807763582,0.974849617998) circle (0.8pt);
\fill ( 0.35449077018110,-0.97484961799) circle (0.8pt);
\fill ( 0.35449077018110,0.974849617998) circle (0.8pt);
\fill ( 0.44311346272637,-0.97484961799) circle (0.8pt);
\fill ( 0.44311346272637,0.974849617998) circle (0.8pt);
\fill ( 0.53173615527165,-0.97484961799) circle (0.8pt);
\fill ( 0.53173615527165,0.974849617998) circle (0.8pt);
\fill ( 0.62035884781693,-0.97484961799) circle (0.8pt);
\fill ( 0.62035884781693,0.974849617998) circle (0.8pt);
\fill ( 0.70898154036220,-0.97484961799) circle (0.8pt);
\fill ( 0.70898154036220,0.974849617998) circle (0.8pt);
\fill ( 0.79760423290748,-0.97484961799) circle (0.8pt);
\fill ( 0.79760423290748,0.974849617998) circle (0.8pt);
\fill ( 0.88622692545275,-0.97484961799) circle (0.8pt);
\fill ( 0.88622692545275,0.974849617998) circle (0.8pt);
\fill ( 0.97484961799803,-0.97484961799) circle (0.8pt);
\fill ( 0.97484961799803,0.974849617998) circle (0.8pt);
\fill ( 1.06347231054331,-0.97484961799) circle (0.8pt);
\fill ( 1.06347231054331,0.974849617998) circle (0.8pt);
\fill ( 1.15209500308859,-0.97484961799) circle (0.8pt);
\fill ( 1.15209500308859,0.974849617998) circle (0.8pt);
\fill ( 1.24071769563386,-0.97484961799) circle (0.8pt);
\fill ( 1.24071769563386,0.974849617998) circle (0.8pt);
\fill ( 1.32934038817914,-0.97484961799) circle (0.8pt);
\fill ( 1.32934038817914,0.974849617998) circle (0.8pt);
\fill ( 1.41796308072441,-0.97484961799) circle (0.8pt);
\fill ( 1.41796308072441,0.974849617998) circle (0.8pt);
\fill ( 1.50658577326969,-0.97484961799) circle (0.8pt);
\fill ( 1.50658577326969,0.974849617998) circle (0.8pt);
\fill ( 1.59520846581496,-0.97484961799) circle (0.8pt);
\fill ( 1.59520846581496,0.974849617998) circle (0.8pt);
\end{tikzpicture}
\end{minipage}
\begin{minipage}{.48\linewidth}
\input{figures/spaces/hom_sample0.4_k400.tex}
\end{minipage}

\begin{minipage}{.48\linewidth}
\input{figures/spaces/hom_sample0.6_k400.tex}
\end{minipage}
\begin{minipage}{.48\linewidth}
\caption{Representation of $\Lambda$ for $k=400$ in the homogeneous example.
Top-left: $\delta = 0.1$. Top-right: $\delta = 0.4$. Bottom-left: $\delta = 0.6$.
The horizontal and vertical distances between dots are $\sqrt{\pi/k}$.}
\label{figure_spaces_homogeneous}
\end{minipage}
\end{figure}

On Figure \ref{figure_convergence_homogeneous}, we present the convergence history
of the method as $\delta$ is increased for different values of $k$. These curves
illustrate the fact that the method converges for any fixed $k$ when increasing
the number of coherent Gaussian states, as predicted by our theory.
Besides, on Figure \ref{figure_cost_homogeneous} we represent the values of
$\delta$ and corresponding number of freedom $N_{\rm dofs}$ required to achieve
a value of about $2 \times 10^{-5}$ (second column of Table \ref{table_homogeneous}) for
different frequencies. Interestingly, the expected rates, namely, $\delta \sim k^{-1/2}$
and $N_{\rm dofs} \sim k^{1/2}$ are observed.

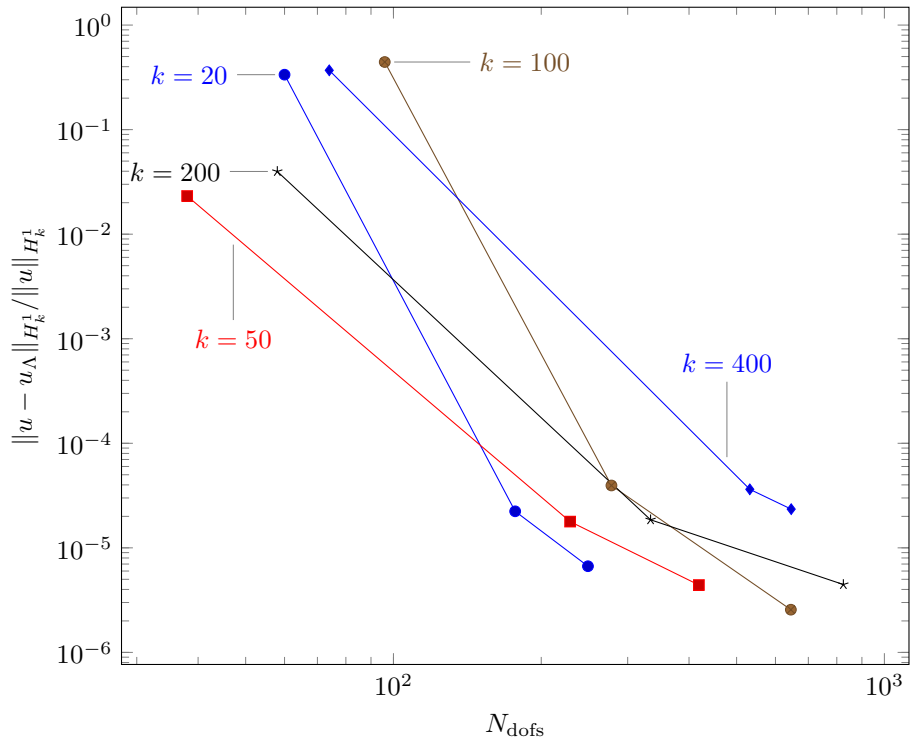
\begin{figure}
\begin{tikzpicture}
\begin{axis}
[
	xlabel = {$N_{\rm dofs}$},
	ylabel = {$\|u-u_\Lambda\|_{H^1_k}/\|u\|_{H^1_k}$},
	ymode = log,
	xmode = log,
	width=.80\linewidth
]

\addplot table[x = ndf, y = error] {figures/data/scattering_homogeneous/k020.txt}
node[pos=0,pin={[pin distance=.5cm]180:{$k=20$}}] {};
\addplot table[x = ndf, y = error] {figures/data/scattering_homogeneous/k050.txt}
node[pos=0.1,pin={[pin distance=1cm]-90:{$k=50$}}] {};
\addplot table[x = ndf, y = error] {figures/data/scattering_homogeneous/k100.txt}
node[pos=0,pin={[pin distance=1cm]0:{$k=100$}}] {};
\addplot table[x = ndf, y = error] {figures/data/scattering_homogeneous/k200.txt}
node[pos=0,pin={[pin distance=.5cm]180:{$k=200$}}] {};
\addplot table[x = ndf, y = error] {figures/data/scattering_homogeneous/k400.txt}
node[pos=0.9,pin={[pin distance=1cm]90:{$k=400$}}] {};

\end{axis}
\end{tikzpicture}
\caption{Convergence histories in the homogeneous example}
\label{figure_convergence_homogeneous}
\end{figure}

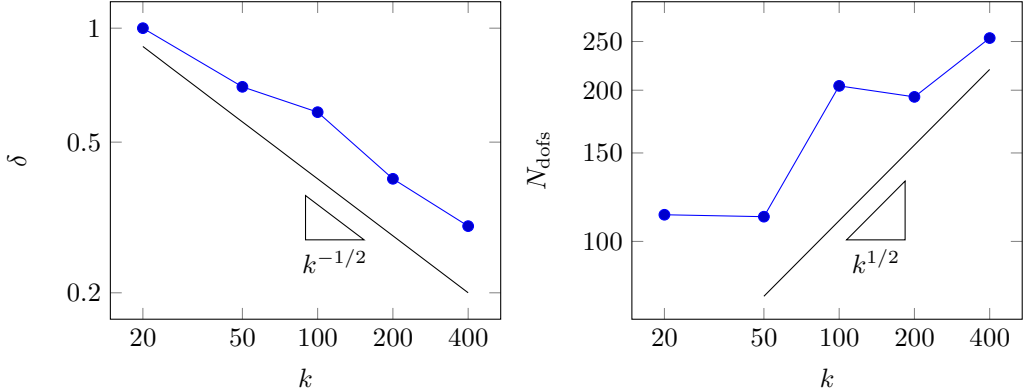
\begin{figure}
\begin{minipage}{.45\linewidth}
\begin{tikzpicture}
\begin{axis}
[
	xlabel = {$k$},
	ylabel = {$\delta$},
	ymode = log,
	xmode = log,
	ytick =       {0.2,0.5,1},
	yticklabels = {0.2,0.5,1},
	xtick       = {20,50,100,200,400},
	xticklabels = {20,50,100,200,400},
	width = \linewidth
]

\addplot table[x = freq, y = delta] {figures/data/scattering_homogeneous/freq.txt};

\plot[black,solid,mark=none,domain=20:400] {4*x^(-1/2)};
\SlopeTriangle{.50}{-.15}{.25}{-0.5}{$k^{-1/2}$}{}

\end{axis}
\end{tikzpicture}
\end{minipage}
\begin{minipage}{.45\linewidth}
\begin{tikzpicture}
\begin{axis}
[
	xlabel = {$k$},
	ylabel = {$N_{\rm dofs}$},
	ymode = log,
	xmode = log,
	ytick =       {100,150,200,250},
	yticklabels = {100,150,200,250},
	xtick =       {20,50,100,200,400},
	xticklabels = {20,50,100,200,400},
	ymin= 70,
	ymax=300,
	width = \linewidth
]

\addplot table[x = freq, y = ndf] {figures/data/scattering_homogeneous/freq.txt};

\plot[black,solid,mark=none,domain=50:400] {11*x^(1/2)};
\SlopeTriangle{.70}{.15}{.25}{0.5}{$k^{1/2}$}{}

\end{axis}
\end{tikzpicture}
\end{minipage}
\caption{Computational cost for a fixed accuracy in the homogeneous example}
\label{figure_cost_homogeneous}
\end{figure}


\subsection{Scattering in an heterogeneous medium}

We now focus on the case where $\alpha = 1$ and
$\mu$ is the only even $C^3$ functions that equals $2$ on $[0,0.7]$,
equals $1$ on $[0.8,+\infty]$ and is a polynomial of degree $7$ on
$[0.7,0.8]$. We select the right-hand side
$f_k(x) \eq k^2(\mu-1)e^{ikx}$.
Here, the analytical solution is not available, and instead,
we rely on a reference solution computed by a Lagrange finite element method
of order $4$ on a grid with $h = 0.02 \cdot k^{-\frac{9}{8}}$ in order to avoid
the pollution effect (see, e.g., \cite{ihlenburg_babuska_1997a}).
The results are listed on Table \ref{table_heterogeneous}, and
Figure \ref{figure_spaces_heterogeneous} represents the phase-space points
included in the space $\Lambda$ for the case $k=200$.

\begin{table}[!h]
\centering
\begin{tabular} {|c|llll|}
\hline\noalign{\smallskip}
 $k$ & \multicolumn{4}{c|}{Relative $H^1_k$ error ($\delta$) [$N_{\rm dofs}$]} \\
\hline\noalign{\smallskip}
 20 & 1.7861e-01 (2.0) [183] & 4.1982e-02 (4.0) [255] & 1.6388e-03 (8.0) [399]  & 1.3463e-04 (12.0) [569]  \\
 50 & 3.0122e-01 (2.0) [337] & 2.5263e-02 (4.0) [511] & 1.0504e-03 (6.0) [695]  & 4.5593e-04  (8.0) [841]  \\
100 & 2.0614e-01 (1.0) [376] & 7.7805e-03 (2.0) [665] & 1.7647e-04 (4.0) [995]  & 1.3300e-04  (6.0) [1309] \\
200 & 3.8799e-01 (0.5) [250] & 8.5455e-03 (1.0) [748] & 6.9635e-04 (2.0) [1233] & 1.2300e-04  (4.0) [1879] \\
\hline
\end{tabular}
\caption{Numerical results for $k=20,50,100,200$ while varying $\delta$}
\label{table_heterogeneous}
\end{table}

We provide the same figures than in the previous experiment, and arrive at similar
observation. First, Figure \ref{figure_convergence_heterogeneous} illustrates the
convergence of the method for fixed values of $k$ as the number of degrees of freedom is
increased. Second, we compare the values of $\delta$ and $N_{\rm dofs}$ to achieve
an accuracy about $10^{-4}$ (last column of Table \ref{table_heterogeneous}) for different
frequencies. In this example too, the expected rates are numerically observed.

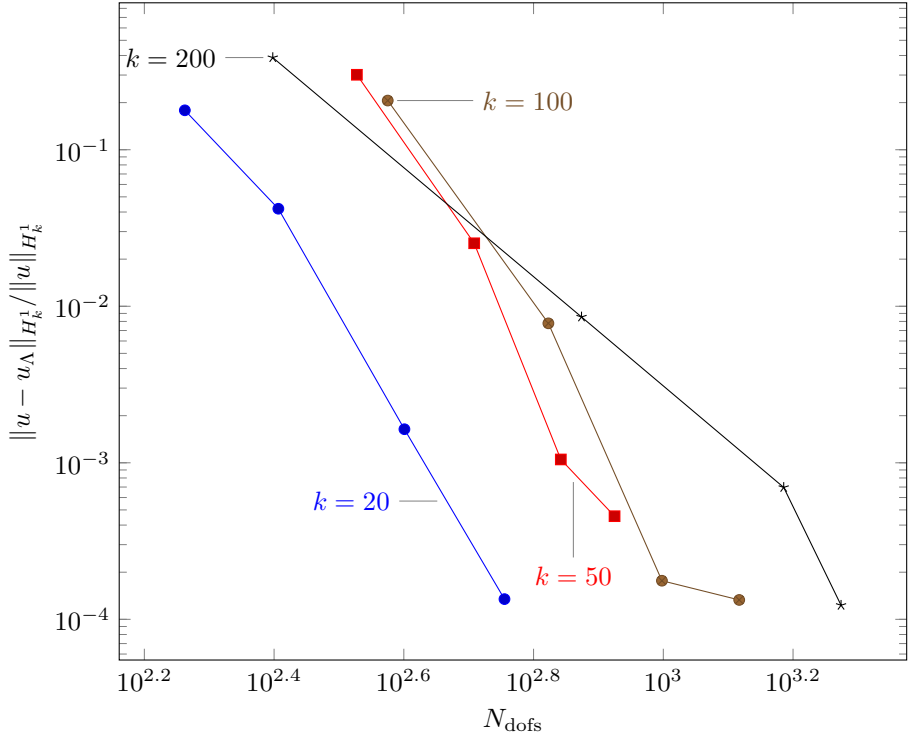
\begin{figure}
\begin{tikzpicture}
\begin{axis}
[
	xlabel = {$N_{\rm dofs}$},
	ylabel = {$\|u-u_\Lambda\|_{H^1_k}/\|u\|_{H^1_k}$},
	ymode = log,
	xmode = log,
	width=.80\linewidth
]

\addplot table[x = ndf, y = error] {figures/data/scattering_heterogeneous/k020.txt}
node[pos=0.8,pin={[pin distance=.5cm]180:{$k=20$}}] {};
\addplot table[x = ndf, y = error] {figures/data/scattering_heterogeneous/k050.txt}
node[pos=0.9,pin={[pin distance=1cm]-90:{$k=50$}}] {};
\addplot table[x = ndf, y = error] {figures/data/scattering_heterogeneous/k100.txt}
node[pos=0,pin={[pin distance=1cm]0:{$k=100$}}] {};
\addplot table[x = ndf, y = error] {figures/data/scattering_heterogeneous/k200.txt}
node[pos=0,pin={[pin distance=.5cm]180:{$k=200$}}] {};

\end{axis}
\end{tikzpicture}
\caption{Convergence histories in the heterogeneous example}
\label{figure_convergence_heterogeneous}
\end{figure}

\begin{figure}
\begin{minipage}{.45\linewidth}
\begin{tikzpicture}
\begin{axis}
[
	xlabel = {$k$},
	ylabel = {$\delta$},
	ymode = log,
	xmode = log,
	ytick =       {4,6,8,12},
	yticklabels = {4,6,8,12},
	xtick       = {20,50,100,200,400},
	xticklabels = {20,50,100,200,400},
	width = \linewidth
]

\addplot table[x = freq, y = delta] {figures/data/scattering_heterogeneous/freq.txt};

\plot[black,solid,mark=none,domain=20:200] {50*x^(-1/2)};
\SlopeTriangle{.50}{-.15}{.25}{-0.5}{$k^{-1/2}$}{}

\end{axis}
\end{tikzpicture}
\end{minipage}
\begin{minipage}{.45\linewidth}
\begin{tikzpicture}
\begin{axis}
[
	xlabel = {$k$},
	ylabel = {$N_{\rm dofs}$},
	ymode = log,
	xmode = log,
	ytick =       {500,1000,1500},
	yticklabels = {500,1000,1500},
	xtick =       {20,50,100,200,400},
	xticklabels = {20,50,100,200,400},
	width = \linewidth
]

\addplot table[x = freq, y = ndf] {figures/data/scattering_heterogeneous/freq.txt};

\plot[black,solid,mark=none,domain=20:200] {100*x^(1/2)};
\SlopeTriangle{.70}{.15}{.35}{0.5}{$k^{1/2}$}{}

\end{axis}
\end{tikzpicture}
\end{minipage}
\caption{Computational cost for a fixed accuracy in the heterogeneous example}
\label{figure_cost_heterogeneous}
\end{figure}


\begin{figure}
\begin{minipage}{.48\linewidth}
\input{figures/spaces/heter_sample0.5_k200.tex}
\end{minipage}
\begin{minipage}{.48\linewidth}
\input{figures/spaces/heter_sample1_k200.tex}
\end{minipage}

\begin{minipage}{.48\linewidth}
\input{figures/spaces/heter_sample2_k200.tex}
\end{minipage}
\begin{minipage}{.48\linewidth}
\caption{Representation of $\Lambda$ for $k=200$ in the heterogeneous example.
Top-left: $\delta = 0.5$. Top-right: $\delta = 1$. Bottom-left: $\delta = 2$.
The horizontal and vertical distances between dots are $\sqrt{\pi/k}$.}
\label{figure_spaces_heterogeneous}
\end{minipage}
\end{figure}

\section{Conclusion}

We propose a new family of finite-dimensional spaces to approximate
the solutions to high-frequency Helmholtz problems. These discretization
spaces are spanned by Gaussian coherent states, which have the
key property to be micro-localised in phase space. This unique feature allows
to carefully select which Gaussian coherent states are included in the discretization,
leading to a frequency-aware discretization space specifically tailored to approximate the
solutions to scattering problems efficiently.

Our key findings correspond to two types of approximability results.
First, assuming for simplicity that the problem is non-trapping, for
general $L^2$ right-hand sides, we show that the Gaussian state approximation
converges and provides a uniform error for all frequencies if the number of degrees
of freedom grows as $(kR)^d$. This result is similar to approximation results
available for finite element discretizations. Our second result applies when
the right-hand side corresponds to a plane-wave scattering problem. In this case, we show
that it is sufficient for the number of degrees of freedom to grow only as
$(kR)^{d-1/2}$ to achieve a constant accuracy for increasing frequencies.
To the best of the authors knowledge, this last estimate suggests that the
proposed discretization space requires substantially less degrees of freedom
than any available method in the literature for high-frequency scattering
problems in general smooth heterogeneous media.

We also present a set of numerical examples where our Gaussian state spaces
are coupled with a least-squares variational formulation. Although the setting
is elementary, these examples successfully illustrate the key features of our
abstract analysis.

While we believe that the proposed results are very
encouraging for a further development of the proposed method, there still
remain several challenges that we would like to address in future works.
First (i), we have chosen to focus on a least-squares method,
because it is simpler to analyse than a Galerkin formulation. However,
least-squares methods are typically poorly-conditioned as compared to their Galerkin counterparts.
While there is no reason to believe that the proposed discrete space would not
work with a Galerkin variational formulation, the analysis appears to be
substantially more complex. Second (ii), our convergence analysis
is currently split into two distinct cases: arbitrary precision for a fixed frequency
(with about $(kR)^d$ DOFs) and or fixed presision for large frequencies (where the
number DOFs scales as $(kR)^{d-1/2}$). In future works, we will focus on unifying
the convergence analysis over the whole frequency range. Third (iii),
especially for three-dimensional problems, computing the entries of the
matrices associated with Gaussian coherent states is not a trivial task, and specific
algorithms should be employed. Besides (iv), although the entries of the discrete matrices
decrease super-algebraically away from their diagonal, they are still dense in principle.
While we are convinced that efficient truncation or compression methods
can be employed, this is still to be analysed. In addition (v),
the resulting linear systems (after possible truncation and/or compression)
will likely have different properties than usual volume methods. As a result, the design
of specific preconditioners will probably be required.
Finally (vi), although we do not see a convenient way to handle
boundary conditions with the proposed Gaussian coherent states, a promising solution
would be to replace the Gaussian by another smooth, but compactly supported, function
in the definition of the Gabor frame.

\appendix

\section{Interplay between Gaussian states and differential operators}
\label{sec:GS}

\begin{lemma}[Action of a differential operator]
Given $\bbeta \in \N^d$ and smooth coefficients $(\TA_{\balpha})_{\balpha \in \N^d}$
that are bounded along with all their derivatives, consider the differential operator
\begin{equation}
\label{eq_diff_operator}
\CP_{\hbar,\bbeta,\TA}
\eq
\sum_{\substack{\balpha \in \N^d \\ \balpha \leq \bbeta}}
\hbar^{[\balpha]} \TA_{\balpha} \partial^{\balpha}.
\end{equation}
Then for all $[\bx_0,\bxi_0] \in \Z^{2d}$ there exists a smooth bounded
function $g_{\hbar,\bbeta,\TA,\bx_0,\xi_0}$ such that
\begin{equation}
\label{eq_diff_action}
(\CP_{\hbar,\bbeta,\TA} \GShknot)(\bx)
=
g_{\hbar,\bbeta,\TA,\bx_0,\bxi_0}(\bx)\GShknot(\bx)
\qquad \forall \bx \in \R^d.
\end{equation}
In addition, we have
\begin{equation}
\label{eq_diff_action_smoothness}
|(\partial^{\bgamma} g_{\hbar,\bbeta,\TA,\bx_0,\xi_0})(\bx)|
\leq
C_{\bbeta,\TA,\bgamma}
\left (
1 + |\bx-\bx_0|^{[\bbeta]} + |\bxi_0|^{[\bbeta]}
\right )
\qquad
\forall \bgamma \in \N^d.
\end{equation}
\end{lemma}

In \eqref{eq_diff_action_smoothness}, the constant $C_{\bbeta,\TA,\bgamma}$ depends on
$\mathrm{A}$ only through bounds on a finite number of derivatives of the functions
$\mathrm{A}_{\balpha}$ for $\balpha \leq \bbeta$,  where this number of derivatives
depends on $\bbeta$ and $\bgamma$. We will use similar notations for constants in the
rest of the appendix.

\begin{proof}
Recalling \cite[Lemma A.2]{chaumontfrelet_ingremeau_2022a}, for $\balpha \in \N$, we have 
\begin{equation*}
(\partial^{\balpha}\GShknot)(\bx)
=
\hbar^{-[\balpha]/2}q_{\balpha}(\hbar^{-1/2}(\bx-\bx_0+i\bxi_0))
\GShknot(\bx),
\end{equation*}
where $q_{\balpha}$ is a polynomial of degree less than or equal to $[\balpha]$.
Hence, we readily see that \eqref{eq_diff_action} holds true with
\begin{equation*}
g_{\hbar,\bbeta,\TA,\bx_0,\bxi_0}(\bx)
\eq
\sum_{\substack{\balpha \in \N^d \\ \balpha \leq \bbeta}}
\hbar^{[\balpha]/2} \TA_{\balpha}(\bx)q_{\balpha}(\hbar^{-1/2}(\bx-\bx_0+i\bxi_0)).
\end{equation*}
To establish \eqref{eq_diff_action_smoothness}, we start with Leibniz' rule
\begin{equation*}
(\partial^{\bgamma} g_{\hbar,\bbeta,\TA,\bx_0,\bxi_0})(\bx)
\eq
\sum_{\substack{\balpha \in \N^d \\ \balpha \leq \bbeta}}
\hbar^{[\balpha]/2}
\sum_{\substack{\bgamma \in \N^d \\ \bdelta \leq \bgamma}}
\left (
\begin{array}{c}
\bdelta \\ \bgamma
\end{array}
\right )
\hbar^{-[\bdelta]/2}
(\partial^{\bgamma-\bdelta}\TA_{\balpha})(\by)
(\partial^{\bdelta} q_{\balpha})(\hbar^{-1/2}(\bx-\bx_0+i\bxi_0)),
\end{equation*}
so that
\begin{equation*}
|(\partial^{\bgamma} g_{\hbar,\bbeta,\TA,\bx_0,\bxi_0})(\by)|
\leq
C_{\bbeta,\TA,\bgamma}
\sum_{\substack{\balpha \in \N^d \\ \balpha \leq \bbeta}}
\hbar^{[\balpha]/2}
\sum_{\substack{\bdelta \in \N^d \\ \bdelta \leq \bgamma}}
\hbar^{-[\bdelta]/2}
|(\partial^{\bdelta} q_{\balpha})(\hbar^{-1/2}(\bx-\bx_0+i\bxi_0))|.
\end{equation*}
Next, we employ the fact that $\partial^{\bdelta} q_{\balpha} = 0$
whenever $\bdelta > \balpha$, and the estimate
\begin{equation*}
|(\partial^{\bdelta} q_{\balpha})(\hbar^{-1/2}(\bx-\bx_0+i\bxi_0))|
\leq
C_{\balpha,\bdelta} \left (
1 + \hbar^{-([\balpha]-[\bdelta])/2}|\bx-\bx_0+i\bxi_0|^{[\balpha]-[\bdelta]}
\right ),
\end{equation*}
leading to
\begin{align*}
|(\partial^{\bgamma} g_{\hbar,\bbeta,\TA,\bx_0,\bxi_0})(\bx)|
&\leq
C_{\bbeta,\TA,\bgamma}
\sum_{\substack{\balpha \in \N^d \\ \balpha \leq \bbeta}}
\hbar^{[\balpha]/2}
\sum_{\substack{\bgamma \in \N^d \\ \bdelta \leq \bgamma}}
\left (1 + (\hbar^{-1/2}|\bx-\bx_0+i\bxi_0|)^{[\balpha]-[\bdelta]}\right )
\\
&\leq
C_{\bbeta,\TA,\bgamma}
\sum_{\substack{\balpha \in \N^d \\ \balpha \leq \bbeta}}
\hbar^{[\balpha]/2}
\left (1 + (\hbar^{-1/2}|\bx-\bx_0+i\bxi_0|)^{[\balpha]}\right )
\\
&\leq
C_{\bbeta,\TA,\bgamma}
\sum_{\substack{\balpha \in \N^d \\ \balpha \leq \bbeta}}
\left (1 + |\bx-\bx_0+i\bxi_0|^{[\balpha]} \right )
\\
&\leq
C_{\bbeta,\TA,\bgamma}
\left (1 + |\bx-\bx_0+i\bxi_0|^{[\bbeta]}\right )
\end{align*}
Then, \eqref{eq_diff_action_smoothness} follows since
\begin{equation*}
|\bx-\bx_0+i\bxi_0|^{[\bbeta]}
\leq
C_{\bbeta} \left ( |\bx-\bx_0|^{[\bbeta]} + |\bxi_0|^{[\bbeta]} \right ).
\qed
\end{equation*}
\end{proof}

For the reader's convenience, we recall the following basic fact.

\begin{proposition}[Moments of a Gaussian]
\label{proposition_gaussian_h_moments}
Consider the Gaussian function
\begin{equation}
\label{eq_gaussian_h}
G_\hbar(\by) \eq \hbar^{-d/2} e^{-\frac{1}{\hbar}|\by|^2}.
\end{equation}
Then, for $\bbeta,\bgamma \in \N^d$, we have
\begin{equation}
\label{eq_gaussian_h_moments}
\int_{\R^d} |\bx^{\bbeta}(\partial^{\bgamma}G_\hbar)(\bx-\bx_0)| \dx
\leq
C_{\bbeta,\bgamma}
\hbar^{-[\bgamma]/2} (\hbar^{[\bbeta]/2} + |\bx_0|^{[\bbeta]})
\qquad
\forall \bx_0 \in \R^d.
\end{equation}
\end{proposition}


We now show that Gaussian states are localised in phase-space.


\begin{lemma}[Quasi orthogonality]
Let $\CP_{\hbar,\bbeta,\TA}$ be as in \eqref{eq_diff_operator}.
Then, for all $\bgamma\in \N^d$, $[\bx_0,\bxi_0] \in \R^{2d}$ and $[\bx_0',\bxi_0'] \in \R^{2d}$,
we have
\begin{equation*}
|(\bx^{\bgamma}\CP_{\hbar,\bbeta,\TA} \GShknot,\GShknotp)|
\leq
C_{\bgamma,\bbeta,\TA,m}
\hbar^{m/2} (1+ |\bx_0 + \bx_0'|^{[\bgamma]})
\frac{1 + |\bxi_0|^{[\bbeta]} + |\bx_0-\bx_0'|^{[\bbeta]}}{1+|\bxi_0-\bxi_0'|^m}
e^{-\frac{1}{4\hbar}|\bx_0-\bx_0'|^2}
\qquad
\forall m \in \N.
\end{equation*}
In particular,
\begin{equation*}
|(\bx^{\bgamma}\CP_{\hbar,\bbeta,\TA} \GShknot,\GShknotp)|
\leq
C_{\bgamma,\bbeta,\TA,m}
\hbar^{m/2} (1 + |\bx_0+ \bx_0'|^{[\bgamma]})
\frac{1 + |\bxi_0|^{[\bbeta]}}{1+|[\bx_0,\bxi_0]-[\bx_0',\bxi_0']|^m}
\qquad
\forall m \in \N.
\end{equation*}
\end{lemma}

\begin{proof}
We have
\begin{equation*}
\GShknot(\bx)\overline{\GShknotp(\bx)}
=
\theta
\hbar^{-d/2}
e^{-\frac{1}{4\hbar}|\bx_0-\bx_0'|^2}
e^{-\frac{1}{\hbar}|\bx-\frac{1}{2}(\bx_0+\bx_0')|^2}
e^{\frac{i}{\hbar}(\bxi_0-\bxi_0') \cdot \bx}
\end{equation*}
where $\theta \in \C$ with $|\theta| = \pi^{-d/2}$.
Recalling \eqref{eq_diff_action} and the notation \eqref{eq_gaussian_h}
from Proposition \ref{proposition_gaussian_h_moments}, we have
\begin{equation*}
\bx^{\bgamma} (\CP_{\hbar,\bbeta,\TA}\GShknot)(\bx)\overline{\GShknotp(\bx)}
=
\theta
\bx^{\bgamma} g_{\hbar,\bbeta,\TA,\bx_0,\bxi_0}(\bx)
e^{-\frac{1}{4\hbar}|\bx_0-\bx_0'|^2}
G_\hbar\left (\bx-\frac{1}{2}(\bx_0+\bx_0')\right )
e^{\frac{i}{\hbar}(\bxi_0-\bxi_0') \cdot \bx}
\end{equation*}
so that
\begin{equation*}
|(\bx^{\bgamma} \CP_{\hbar,\bbeta,\TA}\GShknot,\GShknotp)|
\leq
C e^{-\frac{1}{4\hbar}|\bx_0-\bx_0'|^2} |I|,
\end{equation*}
with
\begin{equation*}
I
\eq
\hbar^{-d/2} \int_{\R^d}
\bx^{\bgamma} g_{\hbar,\bbeta,\TA,\bx_0,\bxi_0}(\bx)
G_{\hbar,\bx_0,\bx_0'}(\bx)
e^{\frac{i}{\hbar}(\bxi_0-\bxi_0') \cdot \bx} \dx,
\end{equation*}
where $G_{\hbar,\bx_0,\bx_0'}(\bx) := G_\hbar\left (\bx-\frac{1}{2}(\bx_0+\bx_0')\right )$.

If $\bxi_0 = \bxi_0'$, this concludes the proof since
\begin{align*}
|I|
&\leq
\hbar^{-d/2} \int_{\R^d}
|\bx^{\bgamma} g_{\hbar,\bbeta,\TA,\bx_0,\bxi_0}(\bx)|
G_{\hbar,\bx_0,\bx_0'}(\bx) \dx
\\
&\leq
C_{\bbeta} \int_\R |\bx|^{[\bgamma]}
(1+ |\bx - \bx_0|^{[\bbeta]}) G_{\hbar,\bx_0,\bx_0'}(\bx) \dx
\\
&=
C_{\bbeta} \int_\R
\left|\by + \frac{1}{2}(\bx_0 + \bx_0')\right|^{[\bgamma]}
\left(1+ \left|\by + \frac{1}{2}\left( \bx_0' - \bx_0\right) \right|^{[\bbeta]}\right)
G_{\hbar}(\by) \dy
\\
&\leq
C_{\bbeta,\bgamma}
(1+ |\bx_0+ \bx_0'|^{[\bgamma]}) (1+|\bx_0-\bx_0'|^{[\bbeta]}).
\end{align*}

If $\bxi_0 \neq \bxi_0'$, then there exists a $j \in \{1,\dots,d\}$ such that
$|(\bxi_0-\bxi_0')_j| \geq 1/\sqrt{d}|\bxi_0-\bxi_0'|$. We integrate $m$
times by part with respect to $\bx_j$, leading to
\begin{align*}
I
&=
\hbar^{-d/2}
\left(\frac{i}{\hbar}(\bxi_0-\bxi_0')_j\right)^{-m}
\int_{\R^d} \partial_j^{m}
(\bx^{\bgamma} g_{\hbar,\bbeta,\TA,\bx_0,\bxi_0}G_{\hbar,\bx_0,\bx_0'})(\bx)
e^{\frac{i}{\hbar}(\bxi_0-\bxi_0') \cdot \bx} \dx
\\
&\leq 
C
\hbar^m |\bxi_0-\bxi_0'|^{-m} I'.
\end{align*}
where
\begin{align*}
|I'|
&\eq
\left |
\hbar^{-d/2}
\int_{\R^d} \partial_j^{m}
(\bx^{\bgamma} g_{\hbar,\bbeta,\TA,\bx_0,\bxi_0}G_{\hbar,\bx_0,\bx_0'})(\bx)
e^{\frac{i}{\hbar}(\bxi_0-\bxi_0') \cdot \bx} \dx
\right |
\\
&=
\left |
\hbar^{-d/2}
\sum_{\ell+\ell'+\ell''=m}
\frac{m!}{\ell!\cdot\ell'\cdot!\ell''!}
\int_{\R^d}
\partial_j^\ell \bx^{\bgamma} (\partial_j^{\ell'} g_{\hbar,\bbeta,\TA,\bx_0,\bxi_0})(\bx)
(\partial_j^{\ell''} G_{\hbar,\bx_0,\bx_0'})(\bx)
e^{\frac{i}{\hbar}(\bxi_0-\bxi_0') \cdot \bx} \dx
\right |
\\
&\leq
C_m 
\sum_{\ell+\ell'+ \ell''=m} I_{\ell,\ell',\ell''}.
\end{align*}
with
\begin{equation*}
I_{\ell,\ell',\ell''}
\eq
\hbar^{-d/2}
\int_{\R^d}
|\partial_j^{\ell} \bx^{\bgamma}|
|(\partial_j^{\ell'} g_{\hbar,\bbeta,\TA,\bx_0,\bxi_0})(\bx)|
|\partial_j^{\ell''} G_{\hbar,\bx_0,\bx_0'}(\bx)| \dx.
\end{equation*}
We then employ \eqref{eq_diff_action_smoothness} and \eqref{eq_gaussian_h_moments}, showing that
\begin{align*}
|I_{\ell,\ell',\ell''}|
&\leq
C_{\bbeta,\TA,m,\bgamma}
\hbar^{-d/2}
\int_{\R^d} (1+ |\bx|^{[\bgamma]})
(1+|\bx-\bx_0|^{[\bbeta]}+|\bxi_0|^{[\bbeta]})
|\partial_j^{\ell''} G_{\hbar,\bx_0,\bx_0'}(\bx)| \dx.
\\
&=
C_{\bbeta,\TA,m,\bgamma} \hbar^{-d/2}
\int_{\R^d}
(1+ |\bx+ \bx_0|^{[\bgamma]})(1+|\bx|^{[\bbeta]}+|\bxi_0|^{[\bbeta]})
\left |\partial_j^{\ell''}
(G_{\hbar})\left (\bx + \frac{1}{2}(\bx_0-\bx_0')\right )
\right |
\dx
\\
&\leq
C_{\bbeta,\TA,m,\bgamma}
\hbar^{-\ell/2} ( 1+ |\bx_0+ \bx_0'|^{[\bgamma]})
\left (
(1+|\bxi_0|^{[\bbeta]})
+
\left |\frac{1}{2}(\bx_0-\bx_0')\right |^{[\bbeta]}\right )
\\
&\leq
C_{\bbeta,\TA,m,\bgamma}
\hbar^{-\ell/2}  ( 1+ |\bx_0+ \bx_0'|^{[\bgamma]}) (1+|\bxi_0|^{[\bbeta]} + |\bx_0-\bx_0'|^{[\bbeta]}).
\end{align*}

\end{proof}

\begin{lemma}[Action of $(P_\hbar-p)$ on Gaussian states]
\label{lemma_residual_L1}
Consider the second-order differential operator
\begin{equation}\label{eq:secondorderop}
P_\hbar
\eq
\hbar^2
\sum_{k,\ell=1}^d
(\BA_\hbar)_{k\ell} \partial_{k\ell} + i\hbar \bb_\hbar \cdot \grad + c_\hbar,
\end{equation}
together with its symbol
\begin{equation*}
p_{\hbar}(\bx,\bxi) = \BA_\hbar(\bx) \bxi \cdot \bxi + \bb_\hbar(\bx) \cdot \bxi + c_\hbar(\bx),
\end{equation*}
where $\BA_\hbar$, $\bb_\hbar$ and $c_\hbar$ are smooth functions that are bounded, along with all their derivatives.
For $\bx_0,\bxi_0 \in \R^d$, we have
\begin{equation}
\label{eq_residual_L1}
\left (
P_{\hbar} - p_{\hbar}(\bx_0,\bxi_0)
\right )
\GShknot
=
r_{\hbar,\bx_0,\bxi_0} \GShknot
\end{equation}
with
\begin{align*}
r_{\hbar,\bx_0,\bxi_0}(\bx)
&\eq
\hbar \tr \BA_{\hbar}(\bx)
\\
&+
i [(\BA_{\hbar}+\BA_{\hbar}^T)(\bx)\bxi_0 + \bb_{\hbar}(\bx)] \cdot (\bx-\bx_0)
\\
&+
(\BA_{\hbar}(\bx)-\BA_{\hbar}(\bx_0)) \bxi_0 \cdot \bxi_0
+
i(\bb_{\hbar}(\bx)-\bb_{\hbar}(\bx_0)) \cdot \bxi_0
+
c_{\hbar}(\bx)-c_{\hbar}(\bx_0)
\\
&+
\BA_{\hbar}(\bx) (\bx-\bx_0) \cdot (\bx-\bx_0).
\end{align*}
In particular, there exists smooth functions $\alpha_{\hbar,\bx_0,\bxi_0,\bbeta}$ such that
\begin{equation}
\label{eq_residual_polynomial_L1}
r_{\hbar,\bx_0,\bxi_0}(\bx)
=
\sum_{\substack{\bbeta \in \N^d \\ [\bbeta] \leq 2}}
\hbar^{(1-[\bbeta])/2}
\alpha_{\hbar,\bx_0,\bxi_0,\bbeta}(\bx) (\bx-\bx_0)^{\bbeta}
\end{equation}
and
\begin{equation}
\label{eq_smoothness_alpha_L1}
\|\partial^{\bgamma} \alpha_{\hbar,\bx_0,\bxi_0,\bbeta}\|_{L^\infty(\R^d)}
\leq
C_{\bbeta,\bgamma} (1+|\bxi_0|^2).
\end{equation}
\end{lemma}

\begin{proof}
We obtain \eqref{eq_residual_L1} by straightforward computations.
We easily identify that
\begin{equation*}
\alpha_{\hbar,\bx_0,\bxi_0,\bzero} = \hbar^{1/2} \tr \BA_{\hbar}
\qquad
\alpha_{\hbar,\bx_0,\bxi_0,\can{j}+\can{\ell}} = \hbar^{1/2} (\BA_{\hbar})_{j\ell}.
\end{equation*}
For the terms with $[\bbeta] = 1$ we further write that for $\varphi \in \{(\BA_\hbar)_{j\ell}, (\bb_{\hbar})_\ell,c_{\hbar}\}$,
there exist smooth functions $\varphi^{\bx_0,k}$ such that
\begin{equation*}
\varphi(\bx)-\varphi(\bx_0) = \sum_{k=1}^d \varphi^{\bx_0,k}(\bx) \cdot (\bx-\bx_0),
\end{equation*}
so that
\begin{equation*}
\alpha_{\bx_0,\bxi_0,\can{k}}
=
\BA_{\hbar}^{\bx_0,k} \bxi_0 \cdot \bxi_0
+
i\bb_{\hbar}^{\bx_0,k} \cdot \bxi_0
+
c_{\hbar}^{\bx_0,k}
+
i
\left (
\sum_{j=1}^d ((\BA_\hbar)_{kj}+(\BA_\hbar)_{jk})\xi_{0,j} + (\bb_\hbar)_k
\right )
\end{equation*}
for $k \in \{1,\dots,d\}$. \qed
\end{proof}

We will now explore further the action of $P_\hbar$ on $\GShknot$,
and show that it is close to $p(\bx_0,\bxi_0) \GShknot$

\begin{lemma}[Action of $(P_\hbar-p_{\hbar})^L$ on Gaussian states]
\label{lem:PSymbol}
For $L \in \N$, we have
\begin{equation*}
\left (P_\hbar - p_{\hbar}(\bx_0,\bxi_0)\right )^L \GShknot
=
r_{\hbar,\bx_0,\bxi_0,L} \GShknot,
\end{equation*}
where $r_{\hbar,\bx_0,\bxi_0,L}$ can be written in the form
\begin{align*}
r_{\hbar,\bx_0,\bxi_0,L}(\bx)
=
\sum_{\substack{\bbeta \in  \Z^d \\ |\bbeta| \leq 2L}}
\hbar^{(L-[\bbeta])/2}
\alpha_{\hbar,\bx_0\bxi_0,L,\bbeta}(\bx)(\bx - \bx_0)^{\bbeta},
\end{align*}
where the functions $\alpha_{\hbar,\bx_0,\bxi_0,\bbeta}$ are smooth and satisfy
\begin{equation}
\label{eq_smoothness_alpha}
\|\partial^{\bgamma} \alpha_{\hbar,\bx_0,\bxi_0,L,\bbeta}\|_{L^\infty(\R^d)}
\leq
C_{L,\bbeta,\bgamma} (1+ |\bxi_0|^2)^L.
\end{equation}
\end{lemma}

\begin{proof}
We shall prove this result by induction. The case $L = 0$ trivially holds with
$\alpha_{\bx_0,\bxi_0,0,\bzero} = 1$. The case $L=1$ is treated in Lemma
\ref{lemma_residual_L1}. Hence, let us assume that $L \in \N$ is such that
\eqref{eq_smoothness_alpha} holds for all $L' \leq L$. We have
\begin{align*}
\left (P_\hbar - p_\hbar(\bx_0,\bxi_0)\right )^{L+1}
\GShknot
&=
\left (P_\hbar - p_\hbar(\bx_0,\bxi_0)\right )
\left (r_{\hbar,\bx_0,\bxi_0,L} \GShknot \right )
\\
&=
\left [P_\hbar - p_\hbar(\bx_0,\bxi_0),r_{\hbar,\bx_0,\bxi_0,L} \right ]
\GShknot
+
r_{\hbar,\bx_0,\bxi_0,L} r_{\hbar,\bx_0,\bxi_0,1} \GShknot,
\end{align*}

Writing the operator $P_{\hbar}$ in the form
\eqref{eq:FormeGeneraleOperateur},
we deduce that 
\begin{align*}
[P_\hbar - p_\hbar(\bx_0,\bxi_0),r_{\hbar,\bx_0,\bxi_0,L}] \GShknot
&=
\hbar^2 \sum_{j,\ell=1}^d
a_{j\ell}
\left (
\frac{\partial^2 r_{\hbar,\bx_0,\bxi_0,L}}{\partial x_j \partial x_\ell}
\GShknot
+
\frac{\partial r_{\hbar,\bx_0,\bxi_0,L}}{\partial x_j}
\frac{\partial \GShknot}{\partial x_\ell} \right )
\\
&+
i\hbar \sum_{j=1}^d
b_j \frac{\partial r_{\hbar,\bx_0,\bxi_0,L}}{\partial x_j} \GShknot
\\
&=
\widetilde{r}_{k,\bx_0,\bxi_0,L+1} \GShknot,
\end{align*}
where 
\begin{equation*}
\widetilde r_{\hbar,\bx_0,\bxi_0,L+1}
\eq
\sum_{j,\ell=1}^d
\left (
\hbar^2 a_{j\ell}\frac{\partial^2 r_{k,\bx_0,\bxi_0,L}}{\partial x_j \partial x_\ell}
+
\hbar \frac{\partial r_{\hbar,\bx_0,\bxi_0,L}}{\partial x_j} (i\xi_{0,\ell} - (x_\ell - x_{0,\ell}))
\right )
+
i\hbar \sum_{j=1}^d b_j \frac{\partial r_{\hbar,\bx_0,\bxi_0,L}}{\partial x_j}.
\end{equation*}
Using the induction hypothesis, we have 
\begin{align*}
\frac{\partial r_{\hbar,\bx_0,\bxi_0,L}}{\partial x_j}(\bx)
&=
\sum_{\substack{\bbeta \in \N^d \\ [\bbeta] \leq 2L}}
\hbar^{(L-[\bbeta])/2}
\frac{\partial \alpha_{\hbar,\bx_0,\bxi_0,L,\bbeta}}{\partial x_j}
(\bx)(\bx-\bx_0)^{\bbeta}
\\
&+
\sum_{\substack{\bbeta \in \N^d \\ [\bbeta] \leq 2L}}
\hbar^{(L-[\bbeta])/2}
\beta_j \alpha_{\hbar,\bx_0,\bxi_0,L,\bbeta}
(\bx)(\bx-\bx_0)^{\bbeta-\can{j}}
\\
&=
\sum_{\substack{\bbeta \in \N^d \\ [\bbeta] \leq 2L}}
\hbar^{(L-[\bbeta])/2}
\frac{\partial \alpha_{\hbar,\bx_0,\bxi_0,L,\bbeta}}{\partial x_j}
(\bx)(\bx-\bx_0)^{\bbeta}
\\
&+
\sum_{\substack{\bbeta \in \N^d \\ [\bbeta] \leq 2L-1}}
\hbar^{(L-1-[\bbeta])/2}
(\beta_j+1) \alpha_{\hbar,\bx_0,\bxi_0,L,\bbeta+\can{j}}
(\bx)(\bx-\bx_0)^{\bbeta}.
\end{align*}

Setting $\alpha_{\hbar,\bx_0,\bxi_0,\bbeta} = 0$ if $[\bbeta] > 2L$, this gives us
\begin{equation*}
\hbar \frac{\partial r_{\hbar,\bx_0,\bxi_0,L}}{\partial x_j}(\bx)
=
\sum_{\substack{\bbeta \in \N^d \\ [\bbeta] \leq 2L}}
\hbar^{(L+1-[\bbeta])/2}
\left (
\hbar^{1/2} \frac{\partial \alpha_{\hbar,\bx_0,\bxi_0,L,\bbeta}}{\partial x_j}(\bx)
+
(\beta_j+1)\alpha_{\hbar,\bx_0,\bxi_0,L,\bbeta+\can{j}}(\bx)
\right )(\bx-\bx_0)^{\bbeta},
\end{equation*}
these terms are multiplied either by
a smooth function, a term of the order $|\bxi_0|$ or a power of
$(\bx-\bx_0)$, which always enters the induction for $L+1$.
The term with the second derivative is treated similarly.

For the remaining term, we simply write that
\begin{align*}
r_{\hbar,\bx_0,\bxi_0,L}r_{\hbar,\bx_0,\bxi_0,1}
&=
\left (
\sum_{\substack{\bbeta' \in \Z^{2d} \\ [\bbeta'] \leq 2}}
\hbar^{(1-[\bbeta'])/2}\alpha_{\hbar,\bx_0,\bxi_0,1,\bbeta'}
(\bx-\bx_0)^{\bbeta'}
\right )
\left (
\sum_{\substack{\bbeta \in \Z^{2d} \\ [\bbeta] \leq 2L}}
\hbar^{(L-[\bbeta])/2}\alpha_{\hbar,\bx_0,\bxi_0,L,\bbeta}
(\bx-\bx_0)^{\bbeta}
\right )
\\
&=
\sum_{\substack{\bbeta  \in \Z^{2d} \\ [\bbeta'] \leq 2}}
\sum_{\substack{\bbeta' \in \Z^{2d} \\ [\bbeta'] \leq 2L}}
\hbar^{(L-[\bbeta])/2 + (1-[\bbeta'])/2}
\alpha_{\hbar,\bx_0,\bxi_0,L,\bbeta}
\alpha_{\hbar,\bx_0,\bxi_0,1,\bbeta'}
(\bx-\bx_0)^{\bbeta+\bbeta'},
\end{align*}
and the result follows. \qed
\end{proof}


\begin{proposition}[Control of $(P_\hbar-p_\hbar)^L$]\label{Prop:OpMoinsSYmbApp}
We have
\begin{equation*}
\begin{aligned}
\|\left (P_\hbar - p_\hbar(\bx_0,\bxi_0)\right )^L \GShknot\|^2
\leq
C_{\BA,\bb,c} (1+|\bxi_0|^2)^L \hbar^L
\end{aligned}
\end{equation*}
for all $[\bx_0,\bxi_0] \in \R^{2d}$.
\end{proposition}

\begin{proof}
We deduce from the previous lemma that
\begin{equation*}
\begin{aligned}
\|\left(P_\hbar - p_\hbar(\bx_0,\bxi_0)\right)^L \Phi_{\hbar,\bx_0,\bxi_0}\|^2
&\leq
C(1+|\bxi_0|^2)^L
(\pi\hbar)^{-d/2}
\int_{\R^d} \left(|\bx-\bx_0|^2 + \hbar \right)^L e^{-\frac{1}{\hbar}|\bx-\bx_0|^2} \dx
\\
&=
C(1+|\bxi_0|^2)^L\pi^{-d/2}
\int_{\R^d} \left( \hbar |\by|^2 + \hbar \right)^L e^{-|\by|^2} \dy
\\
&\leq
C(1+|\bxi_0|^2)^L \hbar^L.
\qed
\end{aligned}
\end{equation*}
\end{proof}

\bibliographystyle{amsplain}
\bibliography{bibliography.bib}

\providecommand{\bysame}{\leavevmode\hbox to3em{\hrulefill}\thinspace}
\providecommand{\MR}{\relax\ifhmode\unskip\space\fi MR }
\providecommand{\MRhref}[2]{%
  \href{http://www.ams.org/mathscinet-getitem?mr=#1}{#2}
}
\providecommand{\href}[2]{#2}
\begin{thebibliography}{10}

\bibitem{Adcock:2019:FNA}
B.~Adcock and D.~Huybrechs, \emph{Frames and numerical approximation}, SIAM
  Review \textbf{61} (2019), no.~3, 443--473.

\bibitem{berenger_1994}
J.~P. B\'erenger, \emph{A perfectly matched layer for the absorption of
  electromagnetics waves}, J. Comput. Phys. \textbf{114} (1994), 185--200.

\bibitem{cessenat2003using}
O.~Cessenat and B.~Despr{\'e}s, \emph{Using plane waves as base functions for
  solving time harmonic equations with the ultra weak variational formulation},
  J. Comput. Acous. \textbf{11} (2003), no.~02, 227--238.

\bibitem{chaumontfrelet_gallistl_nicaise_tomezyk_2018a}
T.~Chaumont-Frelet, D.~Gallistl, S.~Nicaise, and J.~Tomezyk, \emph{Wavenumber
  explicit convergence analysis for finite element discretizations of
  time-harmonic wave propagation problems with perfectly matched layers},
  submitted, preprint \href{https://hal.inria.fr/hal-01887267}{hal-01887267},
  2018.

\bibitem{chaumontfrelet_ingremeau_2022a}
T.~Chaumont-Frelet and M.~Ingremeau, \emph{Decay of coefficients and
  approximation rates in {G}abor {G}aussian frames},
  \href{https://hal.inria.fr/hal-03746979}{hal-03746979}, 2022.

\bibitem{chaumontfrelet_valentin_2020a}
T.~Chaumont-Frelet and F.~Valentin, \emph{A multiscale hybrid-mixed method for
  the {H}elmholtz equation in heterogeneous domains}, SIAM J. Numer. Anal.
  \textbf{58} (2020), no.~2, 1029--1067.

\bibitem{collino_monk_1998a}
F.~Collino and P.~Monk, \emph{The perfectly matched layer in curvilinear
  coordinates}, SIAM J. Sci. Comp. \textbf{19} (1998), no.~6, 2061--2090.

\bibitem{daubechies_grossman_meyer_1986a}
I.~Daubechies, A.~Grossmann, and Y.~Meyer, \emph{Painless nonorthogonal
  expansions}, J. Math. Phys. \textbf{27} (1986), no.~5, 1271--1283.

\bibitem{dorf_2006a}
R.~C. Dorf, \emph{Electronics, power electronics, optoelectronics, microwaves,
  electromagnetics and radar}, Taylor \& Francis, 2006.

\bibitem{dyatlov2019mathematical}
S.~Dyatlov and M.~Zworski, \emph{Mathematical theory of scattering resonances},
  vol. 200, American Mathematical Soc., 2019.

\bibitem{engquist_runborg_2003a}
B.~Engquist and O.~Runborg, \emph{Computational high frequency wave
  propagation}, Acta Numerica \textbf{2003} (2003), 181--266.

\bibitem{Faou:2009:CSQ}
E.~Faou, V.~Gradinaru, and C.~Lubich, \emph{Computing semiclassical quantum
  dynamics with hagedorn wavepackets}, SIAM J. Sci. Comput. \textbf{31} (2009),
  no.~4, 3027--3041.

\bibitem{farhat2001discontinuous}
C.~Farhat, I.~Harari, and L.~P. Franca, \emph{The discontinuous enrichment
  method}, Comput. Methods Appl. Mech. Engrg. \textbf{190} (2001), no.~48,
  6455--6479.

\bibitem{Gabor}
D.~Gabor, \emph{Theory of communication. part 1: The analysis of information},
  Journal of the Institution of Electrical Engineers-Part III: Radio and
  Communication Engineering \textbf{93} (1946), no.~26, 429--441.

\bibitem{galkowski2021perfectly}
J.~Galkowski, D.~Lafontaine, and E.~A. Spence, \emph{Perfectly-matched-layer
  truncation is exponentially accurate at high frequency}, arXiv preprint
  arXiv:2105.07737 (2021).

\bibitem{galkowski_spence_2022a}
J.~Galkowski and E.A. Spence, \emph{Does the {H}elmholtz boundary element
  method suffer from the pollution effect?}, arXiv:2201.09721, 2022.

\bibitem{galkowski2019optimal}
J.~Galkowski, E.A. Spence, and J.~Wunsch, \emph{Optimal constants in
  nontrapping resolvent estimates and applications in numerical analysis}, Pure
  and Applied Analysis \textbf{2} (2019), no.~1, 157--202.

\bibitem{gittelson_hiptmair_perugia_2009a}
C.J. Gittelson, R.~Hiptmair, and I.~Perugia, \emph{Plane wave discontinuous
  {G}alerkin methods: analysis of the $h$-version}, ESAIM Math. Model. Numer.
  Anal. \textbf{43} (2009), 297--331.

\bibitem{Gradinaru:2014:CSW}
V.~Gradinaru and G.A. Hagedorn, \emph{Convergence of a semiclassical wavepacket
  based time-splitting for the {S}chr{\"o}dinger equation}, Numer. Math.
  \textbf{126} (2014), no.~1, 53--73.

\bibitem{Gradinaru:2021:HWS}
V.~Gradinaru and O.~Rietmann, \emph{Hagedorn wavepackets and {S}chr{\"o}dinger
  equation with time-dependent, homogeneous magnetic field}, J. Comput. Phys.
  \textbf{445} (2021), 110581.

\bibitem{Greengard:1987:FMM}
L.~Greengard and V.~Rokhlin, \emph{A fast algorithm for particle simulations},
  J. Comput. Phys. \textbf{73} (1987), no.~2, 325--348.

\bibitem{Gro}
K.~Gr{\"o}chenig, \emph{Foundations of time-frequency analysis}, Springer
  Science \& Business Media, 2001.

\bibitem{Hackbusch:2015:HMM}
W.~Hackbusch, \emph{Hierarchical matrices: algorithms and analysis}, vol.~49,
  Springer, 2015.

\bibitem{hiptmair_moiola_perugia_2011a}
R.~Hiptmair, A.~Moiola, and I.~Perugia, \emph{Plane wave discontinuous
  {G}alerkin methods for the 2d {H}elmhotlz equation: analysis of the
  $p$-version}, SIAM J. Numer. Anal. \textbf{49} (2011), no.~1, 264--284.

\bibitem{hiptmair_moiola_perugia_2016a}
\bysame, \emph{A survey of {T}refftz methods for the {H}elmholtz equation},
  pp.~237--279, Springer International Publishing, 2016.

\bibitem{ihlenburg_babuska_1997a}
F.~Ihlenburg and I.~Babu\v ska, \emph{Finite element solution of the
  {H}elmholtz equation with high wave number. {P}art {II}: {T}he {$h-p$}
  version of the {FEM}}, SIAM J. Numer. Anal. \textbf{34} (1997), no.~1,
  315--358.

\bibitem{imbert2014generalized}
L.-M. Imbert-G{\'e}rard and B.~Despr{\'e}s, \emph{A generalized plane-wave
  numerical method for smooth nonconstant coefficients}, IMA Journal of
  Numerical Analysis \textbf{34} (2014), no.~3, 1072--1103.

\bibitem{imbert2021amplitude}
L.M. Imbert-Gerard, \emph{Amplitude-based generalized plane waves: New
  quasi-trefftz functions for scalar equations in two dimensions}, SIAM Journal
  on Numerical Analysis \textbf{59} (2021), no.~3, 1663--1686.

\bibitem{lafontaine_spence_wunsch_2019a}
D.~Lafontaine, E.A. Spence, and J.~Wunsch, \emph{For most frequencies, strong
  tapping has a weak effect in frequency-domain scattering}, preprint
  \href{https://arxiv.org/abs/1903.12172}{arXiv:1903.12172}, 2019.

\bibitem{lafontaine_spence_wunsch_2022a}
\bysame, \emph{Wavenumber-explicit convergence of the $hp$-fem for the
  full-space heterogeneous {H}elmholtz equation with smooth coefficients},
  Comput. Math. Appl. \textbf{113} (2022), 59--69.

\bibitem{lasser2020computing}
C.~Lasser and C.~Lubich, \emph{Computing quantum dynamics in the semiclassical
  regime}, Acta Numerica \textbf{29} (2020), 229--401.

\bibitem{melenk1996partition}
J.~M. Melenk and I.~Babu{\v{s}}ka, \emph{The partition of unity finite element
  method: basic theory and applications}, Comput. Methods Appl. Mech. Engrg.
  \textbf{139} (1996), no.~1-4, 289--314.

\bibitem{melenk_1999a}
J.M. Melenk, \emph{Operator adapted spectral element methods {I}: harmonic and
  generalized harmonic polynomials}, Numer. Math. \textbf{84} (1999), 35--69.

\bibitem{melenk_sauter_2010a}
J.M. Melenk and S.~Sauter, \emph{Convergence analysis for finite element
  discretizations of the {H}elmholtz equation with {D}irichlet-to-{N}eumann
  boundary conditions}, Math. Comp. \textbf{79} (2010), no.~272, 1871--1914.

\bibitem{melenk_sauter_2011a}
\bysame, \emph{Wavenumber explicit convergence analysis for {G}alerkin
  discretizations of the {H}elmholtz equation}, SIAM J. Numer. Anal.
  \textbf{49} (2011), no.~3, 1210--1243.

\bibitem{monk1999least}
P.~Monk and D.-Q. Wang, \emph{A least-squares method for the {H}elmholtz
  equation}, Comput. Methods Appl. Mech. Engrg. \textbf{175} (1999), no.~1-2,
  121--136.

\bibitem{nonnenmacher2011spectral}
S.~Nonnenmacher, \emph{Spectral problems in open quantum chaos}, Nonlinearity
  \textbf{24} (2011), no.~12, R123.

\bibitem{nonnenmacher2009quantum}
S.~Nonnenmacher and M.~Zworski, \emph{Quantum decay rates in chaotic
  scattering}, Acta mathematica \textbf{203} (2009), no.~2, 149--233.

\bibitem{riou2008multiscale}
H.~Riou, P.~Ladeveze, and B.~Sourcis, \emph{The multiscale {VTCR} approach
  applied to acoustics problems}, J. Comput. Acous. \textbf{16} (2008), no.~04,
  487--505.

\bibitem{sauter_schwab_2010a}
S.A. Sauter and C.~Schwab, \emph{Boundary element methods}, Springer, 2010.

\bibitem{schwab_1998a}
C.~Schwab, \emph{$p-$ and $hp-$ finite element methods}, Clarendon Press, 1998.

\bibitem{tarantola_1984a}
A.~Tarantola, \emph{Inversion of seismic reflection data in the acoustic
  approximation}, Geophysics \textbf{49} (1984), no.~8, 1259--1266.

\bibitem{zworski2012semiclassical}
M.~Zworski, \emph{Semiclassical analysis}, vol. 138, American Mathematical
  Soc., 2012.

\bibitem{zworski2017mathematical}
\bysame, \emph{Mathematical study of scattering resonances}, Bulletin of
  Mathematical Sciences \textbf{7} (2017), no.~1, 1--85.

\end{thebibliography}

\end{document}